\documentclass[11pt]{article}
\usepackage{amsmath, amsthm, amsfonts, amssymb}
\usepackage{url}
\numberwithin{equation}{section}

\newcommand{{\R}}{\mathbb R}
\newtheorem{theorem}{Theorem}[section]

\newtheorem{corollary}[theorem]{Corollary}
\newtheorem{remark}[theorem]{\it \rmfamily Remark}

\newtheorem{example}[theorem]{\it \rmfamily Examples}

\newtheorem{lemma}[theorem]{Lemma}
\newtheorem{hypothesis}{\it \rmfamily Hypothesis}
\newtheorem{proposition}[theorem]{Proposition}

\newcommand{\bth}{\begin{theorem}}
\newcommand{\bpr}{\begin{proposition}}
\newcommand{\epr}{\end{proposition}}

\def\P{{\mathbb P}}
\def\R{{\mathbb R}}

\def\Q{{\mathbb Q}}

\def\lan{{\langle}}
\def\ran{{\rangle}}
\def\hh{{\vskip 0.5 mm \noindent }}

\def\qed{\hfill \hbox{\hskip 6pt\vrule
width6pt height6pt depth1pt  \hskip1pt}
\smallskip}

\begin{document}

\title {
 {\bf   Davie's type uniqueness 
for a class of SDEs with jumps }
}

\author{  Enrico Priola
  \footnote{ E-mail:  enrico.priola@unito.it. 
}
 \\ \\ { Dipartimento di Matematica ``Giuseppe Peano''} \\  
{Universit\`a di Torino}
 \\ { via Carlo Alberto 10,   Torino,  Italy} }

\date{}

\maketitle

\noindent {\bf Abstract:} A  result of A.M. Davie [Int. Math. Res. Not. 2007] 
states that a multidimensional stochastic equation $dX_t = b(t, X_t)\,dt + dW_t$, $X_0=x$, driven by a Wiener process  $W= (W_t)$ with a coefficient  $b$ which is only bounded and measurable has a unique solution  for almost all choices of the driving Wiener path. 
We consider a similar problem when $W$ is replaced by a L\'evy process $L= (L_t)$ and $b$
is $\beta$-H\"older continuous   in the  space variable,
$ \beta \in (0,1)$. 
We 
assume that $L_1$
has a finite 
moment of order $\theta$, for some ${\theta}>0$. Using also a new c\`adl\`ag regularity result for strong solutions, we prove that  strong existence and uniqueness for the SDE together with  $L^p$-Lipschitz continuity of the strong  solution with respect to  $x $ imply a Davie's type uniqueness result for almost all choices of the L\'evy path.
We apply this  result to 
a class of SDEs driven by non-degenerate $\alpha$-stable L\'evy processes, $\alpha \in (0,2)$ and $\beta > 1 - \alpha/2$.

\vspace{2.7 mm}


\vspace{2.7 mm}


\section{Introduction}

In \cite{D} A.M. Davie has proved 
that  a  SDE $dX_t = b(t, X_t)\,dt + dW_t$, $X_0=x \in \R^d$, driven by a Wiener process  $W$ and having a coefficient  $b$ which is only bounded and measurable has a unique solution  for almost all choices of the driving Wiener path. This type of uniqueness is also called 
{\sl path-by-path uniqueness}. 
In other words, 
 adding  a single  path of a Wiener process  $W = (W_t)$ $=
  (W_t)_{t \ge 0}$  regularizes a singular  ODE whose right-hand side $b$ is only  bounded and measurable.  

 We consider a similar uniqueness problem   for SDEs driven by L\'evy noises  with  H\"older continuous    drift term $b$, i.e.,
we deal  with 
\begin{equation} \label{SDE}
 X_{t}(\omega) = x + \int_{s}^{t}b\left(r, X_{r}(\omega)\right)  dr \, + \, L_{t}(\omega) - L_s(\omega), \;\; \; t \in [s,T], 
 \quad
\end{equation}
where  $T>0,$ $s \in [0,T]$,  $x \in {\mathbb R}^d,$ $d \ge 1$,  $b: [0,T] \times {\mathbb R}^d \to {\mathbb R}^d $ is measurable, bounded and
 $\beta$-H\"older continuous in the $x$-variable, uniformly in $t$, $\beta \in (0,1]$.
Moreover   $L=
 (L_t) $ is a 
 $d$-dimensional   L\'evy process defined on a  probability space $(\Omega, {\cal F}, P)$ and $\omega \in \Omega$ (see Section 2; recall that $L_0 =0$, $P$-a.s).
Suppose that   $E [ |L_1|^{\theta} ] < \infty$ for some $\theta >0$
(cf. Hypothesis \ref{zero}). Assuming  that, for any $x \in {\mathbb R}^d$, $s \in [0,T]$,
  strong existence and uniqueness hold for \eqref{SDE}  together with  $L^p$-Lipschitz continuity of the strong solution $(X_t^{s,x})$ with respect to  
$x$, i.e.,  
\begin{equation} \label{nonva} 
\sup_{s \in [0,T]} E \big [\displaystyle{ \sup_{s \le r \le T}} |X_r^{s,x} - X^{s,y}_r|^p \big] \le C \,
 |x-y|^p,\;\;\; x,\, y \in {\mathbb R}^d, \;\; p \in [2, \infty)
\end{equation} 
(cf. Hypothesis \ref{primo} and Section 2) we prove   the following 
result (cf. Theorem \ref{d32})
\begin{theorem} \label{main1}
Assume Hypotheses \ref{primo} and \ref{zero}.
There exists an  event $\Omega' \in {\cal F}$ with $P(\Omega') =1$ such that for any $\omega \in \Omega'$, $x \in {\mathbb R}^d$, the integral equation
\begin{equation} \label{davi}
 f(t) = x + \int_{0}^{t}b\left(r, f(r) + L_{r}(\omega)\right)  dr,\;\;\; t \in [0,T], 
\end{equation} 
has exactly one solution $f$ in $C([0,T];{\mathbb R}^d)$. 
\end{theorem} 
The assumptions and the uniqueness property are clear when $\beta=1$ (the Lipschitz case). When $\beta \in (0,1)$
the result is a special case of assertion (v) in Theorem \ref{d32} which also considers $s \not =0$. It turns out that   
 $f(t) = \phi(0,t,x,\omega) - L_t(\omega) $, $t \in [0,T]$, where 
$( \phi(s,t,x, \cdot))$ is a particular strong solution to \eqref{SDE}.
In Section 6 we will apply the previous theorem to  
a class of SDEs driven by non-degenerate $\alpha$-stable type L\'evy processes, $\alpha \in (0,2)$,
assuming as in \cite{Pr10} that $\beta \in \big(1- \frac{\alpha}{2},1 \big)$.
 Note that 
we can also treat locally H\"older  drifts $b(x)$ 
by a localization procedure (see Corollaries \ref{unb} and \ref{cons}).
 These uniqueness results seem to be new even in dimension one. For instance,  one  can consider  
$$
 dX_t = \sqrt{|X_t|} \, dt + dL_t^{(\alpha)}, \;\; X_0 =x \in \R,
$$
with a symmetric 
 $\alpha$-stable process $L^{(\alpha)} = (L_t^{(\alpha)})$, $\alpha >1$,
and prove that for almost all $\omega \in \Omega$
there exists at most one  
solution for \eqref{davi} with $b(r,x) = \sqrt{|x|}$ and $L = L^{(\alpha)}$.

As already mentioned 
 when $L = W$ is a standard Wiener process, Theorem \ref{main1}  is a special case of Theorem 1.1 in \cite{D}. 
 Recall that  Davie's uniqueness  is  stronger then  the usual pathwise uniqueness  considered in the literature on SDEs   (cf. Remark \ref{patho} and see also \cite{F}).  Pathwise uniqueness deals with solutions which are  adapted stochastic processes and does not consider solutions corresponding to   single paths $(L_t(\omega))_{t \in [0,T]}$.  
 When  $L = W$ 
several results on strong existence and pathwise uniqueness are known for the SDE \eqref{SDE} with very irregular drift $b$: the seminal paper  \cite{Ver}  deals with  $b$ as in the Davie's result;  further recent  results  consider   $b$ which is  only locally in some  $L^p$-spaces (see also \cite{GM}, \cite{KR05} and \cite{FF}).

When 
$L$   is a stable type L\'evy process,  the SDE \eqref{SDE}  with  a H\"older continuous and bounded  drift 
$b$  and  its associated integro-differential generator ${\cal L}_b$ (cf.  \eqref{dee1}) has  received a lot of attention
(see, for instance, \cite{Tanaka}, \cite{Pr10}, \cite{Si},  \cite{SVZ},
 \cite{BDMZ}, \cite{Pr14},
\cite{CSZ} 
and the references therein). On this respect in   Theorem 3.2  of  \cite{Tanaka} the authors proved  that  when $d=1$ 
 and $L$ is a symmetric 
 $\alpha$-stable process, $\alpha \in (0,1)$, pathwise uniqueness 
 may fail      even with a
   $\beta$-H\"older continuous $b$    if
  $\alpha + \beta <1$.   
 
Let us come back to Davie's theorem.
The proof in \cite{D} is self-contained but very  technical;
it relies on explicit computations with Gaussian kernels. An alternative approach to the Davie uniqueness result   has been proposed in 
\cite{Sh} (see in particular Theorems 1.1 and 3.1 in \cite{Sh}).
 This approach uses the flow property of strong solutions of SDEs driven by the Wiener process.
 Beside \cite{D} 
our work   has been inspired by 
 Theorem 3.1 in \cite{Sh} which  deals  with drifts $b$  possibly unbounded in time and such that $b(t, \cdot)$ is H\"older continuous. We mention that applications of  Davie's uniqueness to  Euler approximations for \eqref{SDE} are given  in Section 4 of \cite{D}.

In our proof we use  $L^p$-estimates  \eqref{nonva} which are well-known when $L = W$  (they can be easily deduced from  Section 2  in \cite{FGP}). They are even true for more general drifts $b$   (i.e., $b \in L^{q}(0,T; L^p(\R^d; \R^d))$,  $d/p  + 2/q <1$, $p \ge 2$, $q>2$, see formula (5.9) and  Proposition 5.2 in \cite{FF}). Moreover, when $L$ is a symmetric non-degenerate  $\alpha$-stable process, $b(t,x)=b(x)$,  $\alpha \ge 1$ and $\beta \in (1- \frac{\alpha}{2},1]$,  such estimates  follow   by    Theorem 4.3 in \cite{Pr10}   (see  Theorem  \ref{uno11} for a more general case).

 By the     $L^p$-estimates  \eqref{nonva},  passing through different modifications (see Sections 3 and 4), 
we finally  obtain a   suitable  strong solution  $\phi(s,t,x, \omega)$ (see Theorem \ref{d32}) which  solves  \eqref{SDE} for any $\omega \in \Omega'$, for  some almost sure event  $\Omega'$ which is independent on $s,$ $t$ and $x$.  Such   solution $\phi$ is    used to prove uniqueness of \eqref{davi} (see the proof of (v) of Theorem \ref{d32}).
 We also   establish 
 c\`adl\`ag regularity of $\phi$ with respect to  $s$, uniformly in $t \in [0,T]$ and  $x$, when $x$ varies in compact sets of $\R^d$. This result seems to be new even when $d=1$ and $b$ is Lipschitz continuous  if $L$ is not the Wiener process $W$
(when $L=W$,  the continuous dependence on  $s$, uniformly in $x$, 
has been proved in Section 2 of  \cite{IS} for SDEs with  Lipschitz coefficients). We also prove the continuous dependence of $\phi (s,t, x, \omega)$ with respect to $x$ and the flow property, for any $\omega \in \Omega'$ (see assertions (iii) and (iv) in Theorem \ref{d32}).
There  are   recent papers  on the flow property for solutions to SDEs with jumps (see, for instance, \cite{Pr14}, \cite{LM}, \cite{CSZ} and the references therein). However they do not prove the previous assertions on $\phi$. 
  
Remark that when  $L=W$ and $b(t, \cdot)$ is H\"older continuous as in \eqref{SDE}, proving the existence of  a  regular strong solution like $\phi$
 is   easier.
 Indeed in such case one can use  the well-known Kolmogorov-Chentsov continuity test to get a continuous dependence on $(s,t,x)$. More precisely, when $L=W$, we can apply  the Zvonkin method of \cite{Ver} or the related It\^o-Tanaka trick   of \cite{FGP} and, using    a suitable regular solution $u(t,x)$ of a related Kolmogorov equation (cf. Section 6.2), find that  the process  $\big (u(t, X_t^x)\big)$ solves 
an auxiliary  SDE with Lipschitz continuous coefficients. 
On this auxiliary equation one can perform the Kolmogorov-Chentsov  test as in \cite{Ku84}  and finally obtain the required regular modification of the strong solution.
To get our regular  strong solution $\phi$ we do not pass through 
 an auxiliary  SDE but work directly on \eqref{SDE} using  first a result  in \cite{IS} and then   a  c\`adl\`ag criterion given in \cite{BF}. We apply   this criterion     to a suitable stochastic process with values in a space of  continuous functions defined on $\R^d$ (see Theorem \ref{modi12}). This approach could  be also useful to study regularity properties of solutions to SDEs with multiplicative noise.

In Section 6  we apply Theorem \ref{d32} to a class of SDEs driven by non-degenerate $\alpha$-stable type L\'evy processes, 
 using also results in  \cite{Pr10}  and \cite{Pr14}.
In particular  we prove a Davie's type uniqueness result for \eqref{SDE} when $L$ is a standard rotationally invariant $\alpha$-stable process, $\alpha \in (0,2) $ and $\beta \in (1- \frac{\alpha}{2},1]$. The generator of $L$ is the well-known fractional Laplacian $ -(- \triangle)^{\alpha/2}$.
 To cover  the case $\alpha \in (0,1)$ we also need an analytic   result proved   in \cite{Si} (cf. Remark 5.5 in \cite{Pr14}).
 When  $\alpha \in [1,2)$ and $\beta \in (1- \frac{\alpha}{2},1]$ we can treat more general non-degenerate $\alpha$-stable type processes  like relativistic and truncated stable processes and  some  temperated stable processes (cf. \cite{Pr14} with the references therein and see Examples \ref{233}).
  When 
$\alpha \in [1,2)$  we can also consider  the   singular  $\alpha$-stable  process $L = (L_t)$,   $L_t = (L^1_t, \ldots, L^d_t) $, $t \ge 0$, 
where $L^1$,
 $\ldots, L^d$ are independent
 one-dimensional symmetric $\alpha$-stable
 processes; well-posedness of SDEs driven by this process  has recently received particular attention    (see, for instance, \cite{basschen}, \cite{Pr10}, \cite{Za2},   \cite{Pr14}, \cite{CSZ}).

\section {Notations and assumptions }

We fix basic  notations.  
 We refer to \cite{sato}, \cite{Ku}, \cite{K}  and \cite{A} for more 
details on L\'evy processes with values in ${\mathbb R}^d$. 
By   $\lan x, y \ran $ (or $x\cdot y$)  we denote   the euclidean inner
product between $x$ and $y \in {\mathbb R}^d$, for  $d \ge 1$;
  further 
$|x|$ $ = (\lan x,x\ran)^{1/2}$. If  $H \subset {\mathbb R}^d$
 we denote by $1_H$ its indicator function. The
  Borel $\sigma$-algebra of a Borel set  $C \subset {\mathbb R}^k$, $k \ge 1$, is  indicated by
   ${\cal B}(C)$. Similarly if $(S,d)$ is a metric space we denote  its Borel $\sigma$-algebra by ${\cal B}(S)$.
We consider a complete probability space 
  $(\Omega, {\cal F}, P)$. The expectation with respect to $P$ is indicated with $E.$  
 If ${\cal G} \subset {\cal F}$ is a $\sigma$-algebra,  a  random variable $X: \Omega \to S$  with values in a metric space $(S,d)$ which is measurable from $(\Omega, {\cal G})$ into $(S, {\cal B}(S))$
 is called ${\cal G}$-measurable.
Similarly a function $l : [0,T] \times \Omega \to S$ is ${\cal B}([0,T]) \times {\cal F}$-measurable  if $l$ is measurable with respect to the product $\sigma$-algebra  ${\cal B}([0,T]) \times {\cal F}$. 

In the sequel we often need to specify the possible  dependence of  events of probability one from
 some parameters. Recall that
a set   $\Omega' \subset \Omega$ is  an {\it  almost sure event} if $\Omega' \in {\cal F}$ and $P(\Omega')=1$. To stress that   $\Omega'$ possibly depends also on a parameter $\lambda$ we write ${\Omega}_{\lambda}'$ (the almost sure event $\Omega'_{\lambda}$ may change from one proposition to another);  for instance the notation $\Omega_{s,x}$ 
means that the almost sure event $\Omega_{s,x}$ possibly depends also on $s$ and $x$. 
 We say that a property involving   random variables
 holds on an almost sure event $\Omega'$ to indicate  that such property  holds for any $\omega \in \Omega'$ (i.e.,  such property holds $P$-a.s.).

A $d$-dimensional stochastic process $L=(L_t)$ $= (L_t)_{t \ge 0}$,  $d \ge 1$, defined on $(\Omega, {\cal F}, P)$
 is a  {\it  L\'evy  process}
  if    it  has independent and  stationary  increments,      c\`adl\`ag  paths (i.e., $P$-a.s.,  each mapping $t \mapsto L_t (\omega)$ is c\`adl\`ag from $[0, \infty)$ into ${\mathbb R}^d$; we denote by $L_{s-}(\omega)$ the left-limit in $s>0$)
 and $L_0=0$, $P$-a.s..

Similarly to Chapter II in \cite{Ku84} and  Chapter V in \cite{K} we define for $0 \le s < t < \infty$ the $\sigma$-algebra ${\cal F}_{s,t}^L$ as the completion of the $\sigma$-algebra generated by the random variables $L_r - L_s$, $r \in [s,t]$. We also set ${\cal F}_{0,t}^L = {\cal F}^L_t. $ Since $L$ has independent increments we have that $L_v - L_u$ is independent of ${\cal F}_{u}^L$ for $0 \le  u< v$.
 Note that  $(\Omega, {\cal F}, ({\cal F}_t^L)_{t \ge 0}, P)$ is    an example of 
  stochastic basis
   which satisfies
  the usual assumptions (see
   \cite[page 72]{A}). 
    Given a L\'evy process $L$ there exists a unique function $\psi : {{\mathbb R}^d}  \to {\mathbb C}$ such that 
 $$
 E [e^{i \langle h, L_t\rangle}] = e^{- t
\psi(h)},\,\;  h \in {\mathbb R}^d, \; t \ge 0;
$$
$\psi$ is called the {\it exponent}
 of $L $. The {L\'evy-Khintchine formula } for $\psi$  states that
 \begin{gather} \label{levii}
 \psi(h)= \frac{1}{2}\langle Q h,h  \rangle   - i \langle a, h\rangle
- \int_{{\mathbb R}^d} \! \! \big(  e^{i \langle h,y \rangle }  - 1 -  { i \langle h,y
\rangle} \, {1}_{\{ |y| \le 1\}} \, (y) \big ) \nu (dy), 
\end{gather}
 $h \in {\mathbb R}^d,$ where  $Q$ is a symmetric non-negative definite $d \times d$-matrix, $a
\in {\mathbb R}^d$ and $\nu$ is a $\sigma$-finite (Borel)
measure on ${\mathbb R}^d$, such that  
 ${ \int_{{\mathbb R}^d} (1 \wedge |y|^2 ) \, \nu(dy)}$
$ <\infty,$  $\nu (\{ 0\})=0$ ($1 \wedge |y|^2$ $= \min(1, |y|^2)$);
 $\nu$
 is the {\it  L\'evy  measure} (or intensity measure) of $L.$ 
 The triplet  $(Q, \nu, a)$  uniquely identifies the law of $L$ (see Proposition 9.8 in \cite{sato} or Corollary 2.4.21 in \cite{A}).
It is called {\sl generating triplet} (or characteristics) of the L\'evy process $L$. 

\smallskip
 Given two stochastic processes $X= (X_t)_{t \in [0,T]}$
and $Y = (Y_t)_{t \in [0,T]}$ defined on $(\Omega, {\cal F}, P)$ and with values in a metric space $(S,d)$, we say that $X$ is a {\it modification or version} of $Y$ if for any $t \in [0,T]$, $X_t = Y_t$, $P$-a.s.; if in addition both $X$ and $Y$ have c\`adl\`ag paths then, $P (X_t = Y_t,\; t \in [0,T])
=$ $P (X_t = Y_t,\; \text{for any} \,t \in [0,T])=1.$   

\smallskip Let  $L = (L_t)$ be a $d$-dimensional L\'evy process defined on a { complete probability space}  $(\Omega, {\cal F}, P)$, let $s \in [0,T]$  and $x \in \R^d$ and consider the SDE  
\begin{equation}
\label{SDE2}
dX_t = b(t,X_t)dt + dL_t,\;\;\; s \le  t \le T,\;\;\; X_s =x,
\end{equation} 
 with $b : [0,T] \times \R^d \to \R^d$ which is a locally bounded  Borel function. 

 According to  \cite{Ku84}, \cite{Ku} and \cite{situ} we say that an $\R^d$-valued  stochastic process $U^{s,x}$ $=(U_t^{s,x}) =$ $(U_t^{s,x})_{t \in [s,T]}$
 defined on $(\Omega, {\cal F}, P)$
is a {\it strong solution} to \eqref{SDE2} starting from $x$ at time $s$ if,   for any $t \in [s,T]$, 
the random variable  $U_t^{s,x}: \Omega \to \R^d$ 
 is ${\cal F}_{s,t}^L$-measurable; further we require that there exists an almost sure event $\Omega_{s,x}$  (possibly depending also on $s$ and $x$ but independent of $t$) such that the following conditions hold for any $\omega \in \Omega_{s,x}$:
  (i) the map:
  $t \mapsto U_t^{s,x}(\omega)$ is c\`adl\`ag on $[s,T]$;
  (ii)  
  we have
\begin{equation} \label{SDE3} 
 {U}^{s,x}_t (\omega) =  x + \int_s^t b(r, {U}^{s,x}_r(\omega)) dr + L_{t}(\omega) - L_s(\omega), \;\; t \in   [s , T];
\end{equation}
(iii) 
 the path $t \mapsto L_t(\omega)$ is c\`adl\`ag and $L_0(\omega)=0$.

Given a strong solution $U^{s,x}$ we set
for any $0 \le t \le s$, $U_t^{s,x} =x$ on $\Omega$.

\smallskip 
Let us recall some  function spaces used in the paper.
 We consider  $C_{b}(\mathbb{R}
 ^{d};\mathbb{R}^{k})$, for integers
  $k,\,d\geq1$, as the Banach space  of all continuous and bounded functions
 $g:\mathbb{R}^{d}\rightarrow\mathbb{R}^{k}$  endowed with the
 supremum norm $\| g\|_0$ $= \| g\|_{C_b}$ $ = \sup_{x \in {\mathbb R}^d}|g(x)|,$ $g \in
  C_{b}(\mathbb{R}
 ^{d};\mathbb{R}^{k}).$
 Moreover, $C_{b}^{0,\beta}(\mathbb{R}
 ^{d};\mathbb{R}^{k})$, $\beta \in (0,1]$,
   is the subspace of all $\beta$-H\"older continuous
   functions $g$, i.e., $g$ verifies
$$
 \begin{array}{l}
 [g]_{C^{0,\beta}_b} = [ g]_{\beta}:=\sup_{x\neq x'\in\mathbb{R}^{d}}
 {(|g(x)-g(x')|}\, {|x-x'|^{-\beta}})<\infty
\end{array}
$$
(when $\beta=1$, $g$  is   Lipschitz continuous). 
 If $\beta =0$ we set $C_{b}^{0,0}(\mathbb{R}
 ^{d};\mathbb{R}^{k})$  $= C_{b}(\mathbb{R}
 ^{d};\mathbb{R}^{k})$.
If $\beta \in (0,1)$ we also write $C_{b}^{\beta}(\mathbb{R}
 ^{d};\mathbb{R}^{k}) = C_{b}^{0,\beta}(\mathbb{R}
 ^{d};\mathbb{R}^{k})$; note that 
 $C_{b}^{0,\beta}(\mathbb{R}
 ^{d};\mathbb{R}^{k})$ is a Banach space with the norm $\| \cdot \|_{C^{0,\beta}_b}
  = \| \cdot \|_{\beta}$ $ = \| \cdot \|_0 + [\cdot ]_{\beta}
 $, $\beta \in (0,1].$
 If $\R^k=\R$, we set
 $C_{b}^{0,\beta}({\mathbb{R}}^{d};{\mathbb{R}^k})
 = C_{b}^{0,\beta}({\mathbb{R}}^{d})$ (a similar convention is also used for other function spaces).
 A function $g \in  C_{b}^{}({\mathbb{R}^d};{\mathbb{R}^k}) $ belongs to 
 $C_{b}^{1}({\mathbb{R}}^{d};{\mathbb R}^k) $ if it is differentiable on $\R^d$ and its 
 Fr\'echet derivative $ Dg \in C_{b}({\mathbb{R}}^{d}; \mathbb{R}^{dk})$.
If $\beta \in (0,1)$, a function $g \in   
 C_{b}^{1}({\mathbb{R}}^{d};{\mathbb{R}^k})$ belongs to 
$C_{b}^{1 + \beta}({\mathbb{R}}^{d};{\mathbb{R}^k})$ if $Dg \in C_{b}^{\beta}({\mathbb{R}}^{d}; \mathbb{R}^{dk})$. 
The space $C_{b}^{{1+ \beta}}  
 ({\mathbb{R}}^{d}; \mathbb{R}^{k} )$ is a Banach space  endowed with the norm
 $\| g\|_{1 + \beta}$ $= \| g\|_{C^{1+\beta}_b}$
 $= \| g\|_0$ $+ [Dg]_{\beta}$, $g \in 
 C_{b}^{{1+ \beta}}  
 ({\mathbb{R}}^{d}; \mathbb{R}^{k} )$.  $C^{\infty}_b(\R^d; \R^k)$ is the space of all infinitely differentiable functions from $\R^d$ into $\R^k$ with all bounded derivatives. Finally $g \in C^{\infty}_b(\R^d)$ belongs to
 $ C^{\infty}_0(\R^d)  $ if $g$ has compact support.
 Given a  bounded open set $B \subset \R^d$ we can define similar Banach spaces $C^{\beta}(B)$ and $C^{1+ \beta}(B)$ with norms $\| \cdot \|_{C^{\beta}(B)}$ and  $\| \cdot \|_{C^{1+ \beta}(B)}$, $\beta \in (0,1)$.

We usually require that the drift $b$   belongs to $ L^{\infty}(0, T; C_{b}^{0,\beta}(\R^d; \R^d))$, $\beta \in [0,1]$. This means that 
 $b: [0,T] \times {\mathbb R}^d \to {\mathbb R}^d $ is Borel measurable and bounded, $b(t, \cdot ) \in  C_{b}^{0,\beta}(\R^d; \R^d)$, $t \in [0,T]$, and
   $ [b]_{\beta, T} = \sup_{t \in [0,T]} [b(t, \cdot)]_{C^{0, \beta}_b} < \infty.$

   Set $\| b\|_{\beta,T} = [ b]_{\beta, T}$ $+ \| b\|_0$, 
$\|b\|_0 $ $= \sup_{t \in [0,T], x \in \R^d} |b(t,x)|$
 if $\beta \in (0,1]$ and $\| b\|_{0,T} =  \| b\|_0$, $\beta =0$. 
 Note that
 $(L^{\infty}(0, T; C_{b}^{0, \beta}(\R^d; \R^d))), \| \cdot\|_{\beta,T})$ is a  Banach space.  
 We will  also use
\begin{equation} \label{goo}
{G_0} = C([0,T]; {\mathbb R}^d)
\end{equation}
 to denote 
the  separable Banach space  
consisting of all continuous functions $f: [0,T] \to \R^d$, endowed with the usual supremum norm $\| \cdot\|_{G_0}$. 

Let us formulate our assumptions on  \eqref{SDE} when
  $b \in L^{\infty}(0,T ; C^{0, \beta}_b(
\R^d; \R^d))
$, $\beta \in [0,1]$.  Note that, possibly changing $b(t,x) $ with $b(t,x ) + a$,
to study  the SDE \eqref{SDE} we may always assume that 
in the generating triplet $(Q, \nu,a)$ we have
\begin{equation} \label{d67}
a=0.
\end{equation}
In \eqref{SDE} we  deal with  a
 L\'evy process $L$  defined on $(\Omega, {\cal F}, P)$ and $b \in L^{\infty}(0,T ; $ $C^{0, \beta}_b(
\R^d; \R^d))$
 which satisfy
 \begin{hypothesis} \label{primo} {\em 
 (i) For any $s \in [0,T]$  and $x \in {\mathbb R}^d$
 on $(\Omega, {\cal F}, P)$ 
there exists a strong solution $(U_{t}^{s,x})_{t \in [0,T]}$  to \eqref{SDE2}.

\hh  (ii) Let $s \in [0,T]$.
 Given  any two 
strong
solutions $(U_t^{s,x})_{t \in [0,T]}$
 and $(U_t^{s,y})_{t \in [0,T]}$
 defined on  $(\Omega, {\cal F}, P)$ which both solve   \eqref{SDE2} 
with respect to    $L$ and 
 $b$
 (starting from $x$ and $y \in \R^d$, respectively, at time $s$)
we have, 
for any $p \ge  2$, 
 \begin{equation} \label{ciao221}
 \sup_{s \in [0,T]} E \big [ \displaystyle{ \sup_{s \le t \le T}} \, |\, 
U_t^{s,x} \; - \, U_t^{s,y} |^p \big] 
\le C(T) \,
 |x-y|^p,\;\;\; x,\, y \in {\mathbb R}^d,
 \end{equation}
 with $C(T) \!= \!C \big((\nu, Q, 0)$,  $\| b \|_{\beta,T},
 d, \beta, p, T \big) \!>0\!$   independent of $s$, $x$ and $y$. \qed
}
\end{hypothesis}

\noindent The previous hypothesis holds clearly for any L\'evy process $L$ if $\beta =1$ (the Lipschitz case). Next we  consider    the L\'evy  measure $\nu$ associated to the large {jump parts of $L$}. 
\begin{hypothesis} \label{zero} {\em
There exists ${\theta} >0$ such that 
$\int_{ \{ |x| >1 \} } |x|^{{\theta}} \nu (dx) < \infty.$ \qed
}
\end{hypothesis}
\begin{remark} \label{dw} {\em By Theorems 25.3  and 25.18 in \cite{sato} the following three conditions are equivalent:

\hh (a) $\int_{\{ |x| >1 \} } |x|^{{\theta}} \nu (dx)< \infty$ for some ${\theta} >0$;  \hh (b) $E [|L_t|^{{\theta}}] <  \infty$ for some $t >  0$; 
 \hh (c)  $E [ \sup_{s \in [0,t]} |L_s|^{{\theta}}] <  \infty$ for any $t > 0$.

 Note also that 
$\int_{\{ |x| >1 \} } |x|^{{\theta}} \nu (dx)< \infty$ holds for some ${\theta} >0$ then
$\int_{\{ |x| >1 \} } |x|^{{\theta}'} \nu (dx)< \infty$  for any  ${\theta}' \in (0,{\theta}]$. } 
 \end{remark}

\begin{remark} \label{patho} {\em  We present here for the sake of completeness
some general concepts about solutions 
of SDEs  (cf. \cite{Si} for more details). We will not use these notions in the sequel.  Let  the initial time   $s=0$. 
  A weak solution to \eqref{SDE}  with initial condition $x \in \R^d$ is
 a tuple   $(
\Omega, $ $ {\mathcal F},
 ({\cal F}_t)_{t \ge 0}, P, L,$ $ X) $, where $(
\Omega, {\mathcal F},
 ({\cal F}_t)_{t \ge 0}, P )$ is a stochastic basis
  on which it is defined a   L\'evy process $L$  
  and
 a c\`adl\`ag $({\mathcal F}_{t})$-adapted ${\mathbb R}^d$-valued
process $X = (X_t)  $  which solves \eqref{SDE}  $P$-a.s.. 
A weak  solution $X$ which is $( {\cal F}_t^L)$-adapted 
is called strong  solution.
One say that pathwise uniqueness holds for \eqref{SDE} if given two weak solutions $X$ and $Y$ (starting from  $x \in {\mathbb R}^d$) and defined on the same stochastic basis (with respect to the same   $L$) then $P$-a.s. we have $X_t = Y_t$, for any $t \in [0,T]$.
} 
\end{remark}

\section {Preliminary results on  strong solutions}

 Consider \eqref{SDE2} with $b \in L^{\infty}(0,T; C_b^{0, \beta}(\R^d; \R^d))$, $\beta \in [0,1]$, and
suppose that $L$  defined on $(\Omega, {\cal F}, P)$ and $b$
 satisfy
 Hypothesis \ref{primo}.  
 
Let  $s \in [0,T]$, $x\in {\mathbb R}^d$. We start with a 
 strong solution $({\tilde X}_t^{s,x})_{t \in [0,T]}$ to \eqref{SDE2} defined on $(\Omega, {\cal F}, P)$
and  introduce the  $d$-dimensional process 
 $  {\tilde Y}^{s,x} = (  {\tilde Y}^{s,x}_t)_{t \in [0,T]}$,     
\begin{equation}
\label{yy1}
  {\tilde Y}_t^{s,x} = {\tilde X}_t^{s,x} - (L_t - L_s),\;\; t \ge s.
\end{equation} 
Note that on some almost sure event $\Omega_{s,x}$ (independent of $t$)  we have 
\begin{equation}
\label{yry}
  {\tilde Y}_t^{s,x} = x + \int_s^t b(r,   {\tilde Y}_r^{s,x} + (L_r - L_s)) dr, \;\; t \ge s,
\end{equation} 
 and  $  {\tilde Y}_t^{s,x} = x$ on $\Omega $ if $t\le s$.  It follows that  $(  {\tilde Y}_t^{s,x})_{t \in [0,T]}$ have {\it   continuous paths}.  

Let us fix $s \in [0,T]$ and $x \in \R^d$. We modify the process ${\tilde Y}^{s,x}$ only on  $\Omega \setminus \Omega_{s,x}$ by setting    ${\tilde Y}_t^{s,x}(\omega) =x$, for $t \in [0,T]$, if $\omega \not \in \Omega_{s,x}$ (we still denote by ${\tilde Y}^{s,x}$ such new process).

We find that 
 ${\tilde Y}^{s,x}_{\cdot } (\omega)$ $ \in {G_0} = C([0,T]; {\mathbb R}^d)$, for any $\omega \in \Omega$. 
 Moreover (cf. \eqref{goo}) it is easy to check that  
\begin{equation} \label{ancora}
{\tilde Y}^{s,x} = {\tilde Y}^{s,x}_{\cdot} \;\; \text{ is a random variable with values in 
${G_0} $}.  
\end{equation}
Now, for each fixed $s \in [0,T]$, we will construct a suitable modification of the random field $({\tilde Y^{s,x}})_{x \in {\mathbb R}^d}$  with values in $G_0$.  We need the following  special case of 
Theorem 1.1 of \cite{IS}. It 
 is  
 a generalized 
 Garsia-Rodemich-Rumsey type lemma. 
 \begin{theorem}
 \label{grr} (\cite{IS}) Let $(M,\rho)$ be a separable metric space
 and $(\Omega, {\cal F}, P)$ be a probability space. Let
 $\psi : \Omega \times {\mathbb R}^d \to M$ be a ${\cal F} \times {\cal B}(\R^d)$-measurable map such that $\psi (\omega, \cdot )$ is continuous on $\R^d$, for each $\omega \in \Omega$, and there exists $c>0$ and $p > 2d $ for which
$
E[(\rho(\psi (\cdot, x) ,\psi (\cdot, y))^p]$  $\le c |x-y|^p,$ $  x,y \in {\mathbb R}^d. 
$
 Then, for any $\omega \in \Omega$,  $x,y \in {\mathbb R}^d$,
\begin{equation}
\label{yy}
\rho(\psi (\omega, x) ,\psi (\omega, y)) \le Y(\omega) |x-y|^{1 - \frac{2d}{p}} \, [(|x| \vee |y|)^{\frac{2d + 1}{p}} \vee 1],
\end{equation}  
where  $Y:  \Omega \to [0, \infty]$ is the following $p$-integrable  random variable:
$$
 Y(\omega) =   \,\Big(\int_{{\mathbb R}^d} \int_{{\mathbb R}^d}  \frac{ (\rho(\psi (\omega, x) ,\psi (\omega, y))^p}  {|x-y|^p}
  \;    f(x)  f(y) dx dy\Big)^{1/p},\;\;\; \omega \in \Omega,
$$
with $f(x)  = c(d,p) \big( [|x|^d \, [(\log(|x|) \vee 0)^2] \, \vee 1 \big)^{-1}$, $x \not = 0$, for some constant $c(d,p)>0$.  
\end{theorem}
In Theorem 1.1 of \cite{IS}   $f(x)$ is just defined as $ \big( [|x|^d \, [(\log(|x|) \vee 0)^2] \, \vee 1 \big)^{-1} $. Moreover 
$Y(\omega) = c_3 (\int_{{\mathbb R}^d} \int_{{\mathbb R}^d}  \frac{ (\rho(\psi (\omega, x) ,\psi (\omega, y))^p}  {|x-y|^p}$ $
   f(x)  f(y) dx dy)^{1/p}$.

\begin{lemma} \label{fi}  Consider \eqref{SDE2} with $b \in L^{\infty}(0,T; C_b^{0, \beta}(\R^d; \R^d))$, $\beta \in [0,1]$, and
suppose that $L$  defined on $(\Omega, {\cal F}, P)$ and $b$
 satisfy
 Hypothesis \ref{primo}.  
 Let us fix $s \in [0,T]$  and consider the random field $ {\tilde   
 Y}^{s} = ({\tilde Y}^{s,x})_{x \in {\mathbb R}^d}$ with values in   
 ${G_0}$ (see \eqref{ancora}$)$. We have:

\hh (i) There exists a continuous version $Y^s =({ Y}^{s,x})_{x \in {\mathbb R}^d}$ with values in ${G_0}$ (i.e., for any $x \in {\mathbb R}^d,$ $ Y^{s,x}=   \tilde Y^{s,x}$ in $G_0$ on some almost sure event).

\hh (ii)  For any  $p >2d$ there exists a  random variable $U_{s,p}$ with values in $[0, \infty]$ such that, for any $\omega \in \Omega$, $x, y \in {\mathbb R}^d$,
\begin{equation}
\label{est23}
\|   Y^{s,x}(\omega)   -   Y^{s,y} (\omega)  \|_{G_0} \le U_{s,p}(\omega) \, [(|x| \vee |y|)^{\frac{2d + 1}{p}} \vee 1] \,    |x-y|^{1- 2d/p}.
\end{equation} 
Moreover, with  the same constant $C(T)$ appearing in \eqref{ciao221},
\begin{equation} \label{ri4}
\begin{array}{l}
\sup_{s \in [0,T]} \, E [U_{s,p}^p] \le C(d) \, C(T) < \infty
\end{array}
\end{equation} 
 where $C(d) = \big (\int_{{\mathbb R}^d}
    f(x) dx \big)^2$
(hence $U_{s,p}$  is finite on some almost sure event possibly depending on $s$ and $p$).

\hh (iii) On some almost sure event $\Omega_s'$ (independent of $t$ and $x$) we have
\begin{equation}
\label{yrye}
  {Y}_t^{s,x} = x + \int_s^t b(r,   {Y}_r^{s,x} + (L_r - L_s)) dr, \;\; t \ge s,\; x\in {\mathbb R}^d
\end{equation}
 (where $ {Y}_t^{s,x}(\omega) =  ({Y}_{\cdot}^{s,x}(\omega))(t)$, $t \in [0,T])$. 
\end{lemma}
\begin{proof} {\bf (i)} Using \eqref{ciao221}
 we can apply the   Kolmogorov-Chentsov  continuity test
 as in \cite{Ka}, page 57, and  obtain a continuous version $  Y^s$ of $\tilde Y^s$. The classical proof given in \cite{Ka}  uses the Borel-Cantelli lemma; by such proof it is easy to  show that
  an analogous of \eqref{ciao221} holds for $  Y^s$, i.e., for $p \ge 2,$  $x,y \in {\mathbb R}^d,$ 
\begin{equation}
\label{kol23}
\begin{array}{l}
\sup_{s \in [0,T]} E [ \|  Y^{s,x} -   Y^{s,y}\|_{{G_0}}^p ]
= \sup_{s \in [0,T]} E [ \|  \tilde Y^{s,x} -   \tilde Y^{s,y}\|_{{G_0}}^p ]
\le C(T) |x-y|^p.  
\end{array}
\end{equation} 
{\bf (ii)}  As in Theorem \ref{grr} we consider 
the random variables  
\begin{equation*}
U_{s,p} (\omega) =  \Big(\int_{{\mathbb R}^d} \int_{{\mathbb R}^d} \Big( \frac{\|   Y^{s,x}(\omega)   -   Y^{s,y} (\omega)  \|_{G_0}}{|x-y|}\Big)^p \;    f(x)f(y) dx dy\Big)^{1/p},
\end{equation*}
$\omega \in \Omega,$ $p>2d$ and $s \in [0,T]$. By \eqref{kol23} and
 Theorem \ref{grr}  we  obtain \eqref{est23} and \eqref{ri4}.

\hh {\bf (iii)} We start from equation \eqref{yry} involving the process $(\tilde Y^{s,x})$. Since for some almost sure event $\Omega_{s,x}' \subset \Omega_{s,x}$, we have $ Y_t^{s,x} (\omega)= $ $  \tilde Y^{s,x}_t(\omega)$, $\omega \in \Omega_{s,x}'$, $t \in [0,T]$, we obtain from \eqref{yry} 
$$
 {Y}_t^{s,x} (\omega) = x + \int_s^t b(r,   {Y}_r^{s,x} (\omega) + (L_r (\omega) - L_s (\omega))) dr,
$$
for any  $s \in [t,T]$, $x \in \Q^d$,  $\omega \in 
 \Omega_s' = \bigcap_{x \in  \Q^d } \, \Omega_{s,x}'$. Note also that by (i) the function:  $x \mapsto Y^{s,x}(\omega)$ is continuous for all $\omega \in \Omega$. Take now $x \in \R^d$ and let $(x_n) \subset \Q^d$ be a sequence converging to $x$. It follows from the continuity of $b(r, \cdot)$ and the dominated convergence theorem that, for any $t \ge s$, on $ \Omega_s'$ we have:
\begin{gather*}
{Y}_t^{s,x} = \lim_{n \to \infty} {Y}_t^{s,x_n} =
\lim_{n \to \infty}  {x_n}  + \lim_{n \to \infty}  \int_s^t b(r,   {Y}_r^{s,x_n} + (L_r - L_s)) dr
\\
= x + \int_s^t b(r,   {Y}_r^{s,x} + (L_r - L_s)) dr
\end{gather*}
and this shows the assertion.
\end{proof}
Let $s \in [0,T]$. According to the previous result starting from $Y^s = (Y^{s,x})_{x \in \R^d}$ we  can define   random variables $X_t^{s,x} : \Omega \to {\mathbb R}^d$ as follows: $X_t^{s,x} = x$   if $t \le s$ and
\begin{equation}
\label{yy12}
X_t^{s,x} = Y_t^{s,x} + (L_t - L_s),\;\;\; s,t \in [0,T],\; x \in {\mathbb R}^d, \; s \le t.
\end{equation} 
By the properties of $Y^{s,x}$ 
we get   $P ({\tilde X}^{s,x}_t = {X}^{s,x}_t ,\; t \in [0,T]) =1$, for any $x \in {\mathbb R}^d$ (cf. \eqref{yy1}).
 Moreover, using also  \eqref{yrye}, we find that   for some almost sure event $\Omega_{s}'$ (independent of $x$ and $t$) the map:
$t \mapsto {X}_t^{s,x}(\omega)$ is c\`adl\`ag  on $[0,T]$, for any $\omega \in \Omega_s'$, $x \in \R^d$, and on
$\Omega_s'$ we have
 \begin{equation} \label{ssd}
{X}^{s,x}_t =  x + \int_s^t b(r, {X}^{s,x}_r) dr + L_{t} - L_s,
\; s \le t \le T,\; x \in \R^d.
\end{equation}
 Thus $({X}^{s,x}_t)_{t \in [0,T]}$ is a  particular {\it strong solution} to \eqref{SDE2}. 
By Lemma \ref{fi} we also have, for any $s \in [0,T]$, $x \in {\mathbb R}^d$,  
on $\Omega$  
\begin{equation} \label{doro}
\lim_{y \to x} \sup_{t \in [0,T]}|  
X_t^{s,x} - X_t^{s,y} | =0.
\end{equation}
We can prove the following  flow property.

\begin{lemma} \label{fll}  Under the same assumptions of Lemma \ref{fi}  consider the strong solution   $(X^{s,x}_t)_{t \in [0,T]}$ defined in \eqref{yy12}.  
Let $0 \le s <u \le  T$. There exists an almost sure event $\Omega_{s,u}$ (independent of $t \in [u,T]$ and $x \in \R^d$)  such that 
for $\omega \in \Omega_{s,u}$, $x \in \R^d$, we have 
\begin {equation} \label{co4}
X_t^{s,x } (\omega) =  
 X_t^{u,  \, X_{u}^{s,x}(\omega)} \,(\omega) ,\;\;\; t \in [u,T], \;\; x \in {\mathbb R}^d.
\end{equation}
\end{lemma}
\begin{proof}  Let us fix $s, u \in [0,T]$, $s < u$, and $x \in \R^d$.  We introduce the process $(V_t^x)_{0 \le t \le T}$ on $(\Omega, {\cal F}, P)$ with values in ${\mathbb R}^d$: 
$$
\displaystyle{ V_t^x(\omega)= \begin{cases} 
 X_t^{s,x}(\omega) \;\;\; \;\;\; \text{for} \;\; 0\le t \le u,
\\ 
 X_t^{u,  \, X_{u}^{s,x}(\omega)} \,(\omega)
 \;\;\; \text{for} \;\;  u < t \le T
\end{cases}, \;\;\; \omega \in \Omega.} 
$$
In order to prove \eqref{co4} we will show that 
 $(V_t^x)$ is strong solution to \eqref{SDE2} 
 for $ t\ge s$.  
Then by uniqueness we will get the assertion.
 
It is easy to prove that $(V_t^x)$ has c\`adl\`ag paths.
More precisely, by \eqref{yrye} on some almost sure event $\Omega_{s}' \cap \Omega'_u$ (independent of $x$) we have that 
 $t \mapsto V_t^x(\omega)$ is c\`adl\`ag on $[0,T]$
(note also that, for any $\omega \in \Omega_{s}' \cap \Omega'_u$, $z \in {\mathbb R}^d,$ $\lim_{t \to u^+} X_{t}^{u,z}(\omega) =$ $z$).

Moreover, for any $x \in {\mathbb R}^d$ and $t \ge s$, the random variable $V_t^x$ is ${\cal F}_{s,t}^L$-measurable. The assertion is clear if $t \le u$. Let us consider the case  when $t>u$. 
First $X_{u}^{s,x}$ is  ${\cal F}_{s,t}^L$-measurable.
Define $F_{t,u} (z,\omega) = X_t^{u,  \, z}(\omega)$, $z \in {\mathbb R}^d$, $\omega \in \Omega$. The mapping $F_{t,u}$ is clearly ${\cal B}({\mathbb R}^d) \times {\cal F}_{s,t}^L$-measurable on ${\mathbb R}^d \times \Omega$
  and 
 $F_{t,u}(\cdot , \omega)$ is continuous on ${\mathbb R}^d$, for any $\omega \in \Omega$, by \eqref{doro}. It follows  that also the map: $\omega \mapsto  F_{t,u} (X_{u}^{s,x}(\omega), \omega )$ is ${\cal F}_{s,t}^L$-measurable.

It is clear that $(V_t^x)$ solves \eqref{ssd} on  $\Omega_{s}'$ when $s \le t \le u$ (recall    
 \eqref{yrye}). Let us consider the case when $t \ge u$. 
According to \eqref{ssd} we know  that on 
 $\Omega_{u}'$ we have 
\begin{equation} \label{ssd2}
{X}^{u, {X}^{s,x}_u}_t =  {X}^{s,x}_u  + \int_u^t b(r, 
{X}^{u, {X}^{s,x}_u}_r ) dr + L_{t} - L_u, \;\; t \ge u.
\end{equation} 
Hence  on  $\Omega_{u}' \cap \Omega_s'$ we have for $t \ge u$
\begin{gather*}
V_t^x = {X}^{u, {X}^{s,x}_u}_t =  x + \int_s^u b(r, {X}^{s,x}_r) dr + L_{u} - L_s
\\
+ \int_u^t b(r, 
{X}^{u, {X}^{s,x}_u}_r ) dr + L_{t} - L_u
 = x  + \int_s^t b(r, 
{V}^x_r ) dr + L_{t} - L_s.
\end{gather*}
It follows that  $(V_t^x)$ solves \eqref{ssd} on  $\Omega_{s}' \cap \Omega'_u$ when $s \le t \le T$. By Hypothesis \ref{primo} we infer that, for any $x \in {\mathbb R}^d$, on some almost sure event $\Omega_{s,u,x}$ we have  that  $V^{x}_t = {X}^{s,x}_t$, $t \in [s,T]$. 
In  particular we get 
$V^{x}_t = {X}^{s,x}_t$, $t \in [u,T]$ and this proves \eqref{co4} at least on an almost sure event $\Omega_{s,u,x}$.

To remove the dependence on $x$ in the almost sure event, we note that the mapping:  
 $ x  \mapsto V^{x}_{t}(\omega)  $
is   continuous from  ${\mathbb R}^d$ into ${\mathbb R}^d$,
for any $\omega \in \Omega$, $t \in [0,T]$ (see \eqref{doro}).  Arguing as in the final part of the proof of Lemma \ref{fi} we obtain that 
 $X^{s,x}_t (\omega) = V^{x}_t(\omega),$ for $ t \in [u,T],$   $x\in {\mathbb R}^d$ and $ \omega \in 
 \Omega_{s,u} = \bigcap_{x \in \Q^d } \Omega_{s,u,x}$. This proves \eqref{co4}. 
 \end{proof}
Following \cite{royden} page 169 (see also Problem 48 in \cite{royden}) we introduce  the space
  $ C({\mathbb R}^d; {G_0})$
 consisting of all continuous functions from ${\mathbb R}^d$ into ${G_0}= C([0,T]; {\mathbb R}^d)$ endowed with the compact-open topology (or the topology of the uniform convergence on compact sets).  
This is a complete metric space  endowed with the following metric: 
 \begin{equation} \label{metr}
d_0(f,g) = \sum_{N \ge 1} \frac{1}{2^N} \frac{\sup_{|x|\le N} \|f(x) - g(x) \|_{G_0} }{1+ \sup_{|x|\le N} \|f(x) - g(x) \|_{G_0} },\;\;\; f,g \in C({\mathbb R}^d; {G_0}).
\end{equation}
 It is well-know that $ C({\mathbb R}^d; {G_0})$ is also  \textit{separable} (see, for instance, \cite{khan};  on the other hand $C_b({\mathbb R}^d; {G_0})$ is not separable). We will also consider  the following  projections  
\begin{equation} \label{pro2} 
 \pi_x : C({\mathbb R}^d; {G_0}) \to {G_0}, \;\;\ \pi_x (f) = f(x) \in G_0, \;\; x \in {\mathbb R}^d, \;\; 
 f \in C({\mathbb R}^d; {G_0})
\end{equation}
 (each $\pi_x$ is a continuous map).
According to Lemma \ref{fi} for any $s \in [0,T]$  the random field $(Y^{s,x})_{x \in {\mathbb R}^d}$ has continuous paths. 
It is not difficult to prove that, 
 for any $s \in [0,T]$, the mapping:
\begin{equation}
\label{cte}
 \omega  \mapsto  Y^s (\omega) = Y^{s,  \, \cdot \,} (\omega)
\end{equation} 
is measurable from $(\Omega, {\cal F}, \P) $ with values in $C({\mathbb R}^d; {G_0})$. Indeed thanks to the  separability of $C({\mathbb R}^d; {G_0})$ to check the  measurability  it is enough to prove that counter-images of balls 
 $B_r(f_0) = \{ f \in C({\mathbb R}^d; {G_0})  : \, 
  \sum_{N \ge 1} \frac{1}{2^N} \frac{\sup_{\{ |x|\le N , x \in \Q^d\} } \|f(x) - f_0(x) \|_{G_0} }{1+ \sup_{\{ |x|\le N , x \in \Q^d\}} \|f(x) - f_0(x) \|_{G_0} } <r \}$, $r>0, $ $f_0 \in C({\mathbb R}^d; {G_0})$, are events in $\Omega$.  

\smallskip
In the sequel we will set  $Y = (Y^s)_{s \in [0,T]}$ to denote the previous stochastic
process with values in $C({\mathbb R}^d; {G_0})$ and defined on $(\Omega, {\cal F}, P)$.

\section{
A version  of the  solution  which    is c\`adl\`ag with respect to the initial time $s$ }

 In Theorem \ref{modi12} we  will prove the existence of a c\`adl\`ag modification $Z$ of the process $Y= (Y^s)_{s \in [0,T]}$ with values in $C({\mathbb R}^d; {G_0})$ (cf. \eqref{cte}).  
 In particular $Z$
 is
   a  modification of $Y$ which is c\`adl\`ag in $s$ uniformly in $x$, when $x$ varies on compact sets of ${\mathbb R}^d$. In Lemma \ref{le1} we will study important properties of $Z.$
Before discussing on c\`adl\`ag modifications
we  recall a standard  definition. 

A process $X = (X_t)_{t \in [0,T]}$ defined on $(\Omega, {\cal F}, P)$ with values in a  metric  space $(S,d)$ is
 {\it stochastically continuous  (or continuous in probability)}   if for any $t_0 \in [0,T]$, $X_t$ converges to $X_{t_0}$
in probability (see \cite{GS} for more details).
  
Important results on c\`adl\`ag modifications for stochastic processes were  given by  Gikhman and Skorokhod (see Section III.4 in \cite{GS}). We will use a recent  result given in Theorem 4.2 of \cite{BF}. In contrast with \cite{GS} the proof of this theorem does not require  the separability of the stochastic process.  It is stated in \cite{BF} for stochastic processes $(X_t)$ when $t \in [0,1]$. However a simple rescaling argument shows that it holds when  $t \in [0,T]$, for any $T>0$.

\begin{theorem} \label{be} (\cite{BF}) Let $X = (X_t)_{t \in [0,T]}$
be a stochastically continuous  process defined on a complete probability space and with values in a complete metric space $(S,d)$.
Let $0 \le s < t < u \le T$ and define 
$\triangle (s,t,u) =  d(X_s, X_t) \wedge d(X_t, X_u) 
$. 
A sufficient assumption in order  that   $X$ has a modification with c\`adl\`ag paths is the following one:
there exist non-negative real functions  $\delta $ and $x_0$ ($\delta $ is   non-decreasing and continuous on $[0,T]$, $\delta(0)=0$, and $x_0$ is decreasing and integrable on $(0,T]$) such that the following conditions hold, for any $
0 \le s < t < u \le T$, $M>0$, 
\begin{gather} \label{bez1}
E [\triangle (s,t,u)  1_{\triangle (s,t,u)  \ge M}]
 \le \delta (u-s) \, \int_0^{P (\triangle (s,t,u) \ge M)}
 x_0 (r)dr, 
\\ 
\label{bez2}
 \int_0^1 \Big (  u^{-1} \int_0^u x_0 (r)dr \Big) \, \frac{\delta (u)}{u} du < \infty.
\end{gather} 
\end{theorem}
The next result follows easily (cf. Section III.4 in \cite{GS}).
\begin{corollary} \label{be1}  Let $X = (X_t)_{t \in [0,T]}$
be a stochastically continuous  process with values in a complete metric space $(S_,d)$. A sufficient condition in order that  $X$ has a  c\`adl\`ag modification is the following one:  there exists $q>1/2$ and  $r>0$
such that, for any $
0 \le s < t < u \le T$, 
  we have
\begin{gather} \label{bez11}
E \big [d(X_s, X_t)^q \, \cdot  \, d(X_t, X_u)^q  ]
 \le 
 C |u-s|^{1+r}. 
 \end{gather} 
\end{corollary}
\begin{proof} 
In order to apply Theorem \ref{be} we introduce $x_0(h) = \frac{2q -1}{2q}\, h^{-1/2q} $, $h \in (0,T]$.
 Let us fix $0 \le s < t < u$ and $M>0$.
Noting that for $a,b \ge 0$ we have $a \wedge b \le \sqrt{a} \, \sqrt{b}$. We find by the H\"older inequality 
\begin{gather*}
 E [\triangle (s,t,u)  1_{\triangle (s,t,u)  \ge M}]
\le 
( E [\triangle (s,t,u)^{2q}] )^{1/2q}  \,  (P (\triangle (s,t,u)  \ge M ))^{\frac{2q-1}{2q}}
\\
\le E \big [d(X_s, X_t)^q \, \cdot  \, d(X_t, X_u)^q \big] 
\big)^{1/2q} \, \int_0^{P (\triangle (s,t,u) \ge M)}
 x_0 (r)dr \\ \le C^{1/2q} \, |u-s|^{(1+ r)/2q} \, \int_0^{P (\triangle (s,t,u) \ge M)}
 x_0 (r)dr.
\end{gather*}
Setting $\delta (h) = h^{(1+ r)/2q}$, $h \in [0,T]$, we see that $\int_0^T   {\delta (u)}\, {u^{-1 - \frac{1}{2q}}} du 
< \infty$ 
 is equivalent to \eqref{bez2};  we get the assertion.
\end{proof}
We now prove the stochastic continuity of $Y$.
\begin{lemma}  \label{modi1a} 
Consider \eqref{SDE2} with $b \in L^{\infty}(0,T; C_b^{0, \beta}(\R^d; \R^d))$, $\beta \in (0,1]$, and
suppose that $L$  and $b$
 satisfy
 Hypotheses \ref{primo} and \ref{zero}. 
 Then  
 the  process $Y= (Y^s)$ with values in 
$C({\mathbb R}^d; {G_0})$ (see (\ref{cte})) is continuous in probability. 
 \end{lemma}
\begin{proof} 
 Let us fix $s \in [0,T]$. We have to prove that
\begin {equation} \label{te25}
\lim_{s' \to s} P \Big( \sup_{|x|\le N}\sup_{t \in [0,T]}| Y_{t}^{s,x} -  Y_{t}^{s',x}| > r \Big) = 0, \;\; \text{for any $r>0$, $N \ge 1$.}
\end{equation}
Indeed this is equivalent to $\lim_{s' \to s} P \big( 
d_0 (Y^{s} ,  Y^{s'}) > r \big) = 0,$ $r>0$. 
 To this purpose it is enough to check both the left and the right continuity in \eqref{te25}. Let us check the right continuity in $s$ (assuming $s \in [0,T)$). The proof of the left-continuity in $s$ can be done in a  similar way.
Since $C^{0, \beta}_b(\R^d; \R^d) \subset 
 C^{0,\beta'}_b(\R^d; \R^d)$ for $0 < \beta' \le \beta \le 1$ we may suppose that $\beta $ is sufficiently small; we will assume (cf. Hypothesis \ref{zero})
\begin{equation} \label{bet}
 \beta (2d + 1)\,  < 2d {\theta}.
\end{equation}  
 Let $(s_n) \subset ]s,T]$ with  $s_n \to s$. We have to prove  
 that for  fixed $N \ge 1$, $\delta >0$,
 \begin {equation} \label{te2}
\lim_{n \to \infty} P \Big(\sup_{|x|\le N} \sup_{t \in [0,T]}| Y_{t}^{s,x} -  Y_{t}^{s_n,x}| > \delta \Big) = 0.
\end{equation}
If we show that 
\begin{equation} \label{de33}
E \Big [\sup_{0 \le t \le T} \sup_{|x|\le N}| Y_{t}^{s,x} -  Y_{t}^{{s_n},x}| \Big]   \to 0
 \;\; \text{ as $n \to {\infty}$.}
\end{equation} 
then \eqref{te2} follows. Let us fix $n \ge 1$ and consider the random variable $J_{t,x,n, s} =| Y_{t}^{s,x} -  Y_{t}^{{s_n},x}|$.
If $t \le s$
 we find $J_{t,x,n, s} =0$. If ${s} \le t \le  s_n $  then, for any $x \in {\mathbb R}^d$,  on some almost sure event $\Omega_{s,s_n}$
 (independent of $x$ and $t$; see  \eqref{yrye})
$$
 J_{t,x,n, s} =  \Big|\int_{{s}}^{t} b(r,Y_r^{{s_n} , x} + (L_r - L_{{s}}))dr  \Big| \le \| b\|_0  |t-{s}| \le \| b\|_0  |s -{s_n}|.
$$
Hence in order to get \eqref{de33}  we  need to prove that  
\begin{equation} \label{de3}
E \Big [\sup_{s_n \le t \le T} \sup_{|x|\le N}| Y_{t}^{s,x} -  Y_{t}^{{s_n},x}| \Big]   \to 0
 \;\; \text{ as $n \to {\infty}$.}
\end{equation} 
Let $t \ge {s_n}$. We have on   $\Omega_{s,s_n}$ 
\begin{gather} \label{dir2}
\sup_{|x| \le N} | Y_{t}^{s,x} -  Y_{t}^{{s_n},x}|  \le \sup_{|x| \le N} \Big |\int_{s}^{t} b(r,X_r^{s , x})dr - \int_{{s_n}}^{t} b(r,X_r^{{s_n} , x})dr \Big|
\\ 
\nonumber \le 2 |s-{s_n}| \, \| b\|_{0} +    
  \sup_{|x| \le N}  \int_{{s_n}}^{t} |b(r,X_r^{s,x} )  -  b(r,X_r^{{s_n},x} )  |
  dr.  
\end{gather}
By Lemma \ref{fll} on some almost sure event $\Omega'_{s,s_n} \subset \Omega_{s, s_n}$ (independent of $x$ and  $r$) we have for $r \in [s_n, T]$
\begin{gather*}
\sup_{|x| \le N} \,  | b(r,X_r^{s,x} )  -  b(r,X_r^{{s_n},x} )| =  
 \sup_{|x| \le N} \, | b(r, X_r^{{s_n}, X_{{s_n}}^{s,x}})    - b(r, X_r^{{s_n},x})  |  
\\
 \le [b]_{\beta,T} \sup_{|x| \le N} \, \sup_{r \in [0,T]}| X_r^{{s_n}, X_{{s_n}}^{s,x}}    -  X_r^{{s_n},x}  |^{\beta}
 =  [b]_{\beta,T} \sup_{|x| \le N}\|  Y^{s_n, X_{{s_n}}^{s,x}}   -  Y^{s_n,x}   \|_{G_0}^{\beta}.
\end{gather*}
By Lemma \ref{fi} with $p= 4d$, setting  $U_{s'}
 = U_{s',p}$, $s' \in [0,T]$, we get 
 \begin{gather} \label{deb2b}
\sup_{|x| \le N} \,  | b(r,X_r^{s,x} )  -  b(r,X_r^{{s_n},x} )|
\\ \nonumber 
\le [b]_{\beta, T} \,  [(|x| \vee |X_{{s_n}}^{s,x}|)^{\frac{2d + 1}{4d}} \vee 1]^{\beta}\,  U_{s_n}^{\beta} \, \sup_{|x| \le N} \, |x-X_{{s_n}}^{s,x}|^{{\beta/2}}.
\end{gather}
 Noting that, for $|x| \le N$, $n \ge 1$, 
 $|X_{{s_n}}^{s,x}| $ $ \le  N + 2T \| b\|_0 + |L_{s_n} - L_s|$
 we obtain on $ \Omega'_{s,s_n}$
\begin{gather} \label{deb2}
\sup_{|x| \le N} \,  | b(r,X_r^{s,x} )  -  b(r,X_r^{{s_n},x} )|
 \le [b]_{\beta, T} \,   V_{s, s_n, N}^{\beta} \, \sup_{|x| \le N} \, |x-X_{{s_n}}^{s,x}|^{{\beta/2}},
\end{gather}
 $r \in [s_n,T]$,
where we have introduced the random variables 
\begin{equation} \label{vvv}
V_{s, s', N} = \Big[ N^{ \frac{ 2d + 1 }{4d} } + (2T \| b\|_0)^{\frac{2d + 1 }{4d}} + |L_{s'} - L_s|^{\frac{2d + 1 }{4d}} \Big]\,  U_{s' },
\end{equation}
 $0 \le s< s' \le T$.
By Remark \ref{dw} and \eqref{bet} we know that, for any $n \ge 1$,
$$
E [|L_{s_n} - L_s|^{\frac{\beta (2d + 1) }{2d}}] 
= E [ |L_{{s_n} - s} |^{\frac{\beta (2d + 1) }{2d}} ]
\le E [ \sup_{s \in [0,T]} |L_s|^{\frac{\beta (2d + 1) }{2d} }]
< \infty, 
$$
since $E [ \sup_{r \in [0,T]} |L_r|^{{\theta}}] < \infty.$
Using also that 
 $\sup_{r \in [0,T]} E [U_{r,p}^{2\beta}] = k'  < \infty $ (see \eqref{ri4})  we  obtain by the Cauchy-Schwarz inequality  
\begin{equation}
\label{sup2}
\sup_{0 \le s < s' \le T} E [V_{s,s', N}^{\beta}] = k_0  < \infty
\end{equation}  
($k_0$ also depends on $N$).  Let us  revert  to
 \eqref{deb2}. Since
\begin{equation} \label{do2}
|X_{{s_n}}^{s,x} -x| \le \int_{s}^{{s_n}} | b(r,X_r^{s,x} )| dr + 
 |L_{{s_n}} - L_s| \le \| b \|_0 |s-s_n| +  |L_{{s_n}} - L_s|, 
\end{equation}
 for any $x \in {\mathbb R}^d$, $n \ge 1,$  we obtain for $r \in [s_n, T]$
\begin{equation} \label{deb24}
\sup_{|x| \le N} \,  | b(r,X_r^{s,x} )  -  b(r,X_r^{{s_n},x} )|
\le [b]_{\beta, T}  \,  V_{s,s_n, N}^{\beta} \, 
  ( \| b\|_0 \, |s- s_n |  + |L_{{s_n}} - L_s|)^{\beta/2}
\end{equation}
and so (cf. \eqref{dir2})
$$
  \sup_{|x| \le N}  \int_{{s_n}}^{t} |b(r,X_r^{s,x} )  -  b(r,X_r^{{s_n},x} )  | dr 
\le T  [b]_{\beta, T}  \,  V_{s,s_n, N}^{\beta}  
  ( \| b\|_0  |s- s_n |  + |L_{{s_n}} - L_s|)^{\beta/2}.
$$
Let us define the random variables  $Z_n = \| b\|_0 \, |s- s_n | + |L_{{s_n}} - L_s|
$. By the stochastic continuity of $L$ we know that
\begin{equation}
 \label{lle}
\lim_{n \to \infty}  P \big( Z_n   > \delta \big) =0, \;\;\; \delta>0.
\end{equation} 
 Using  \eqref{dir2} 
on an almost sure event $\Omega_{s,s_n}'$, for any $\delta >0$, we have
\begin{gather*}
\sup_{s_n \le t \le T } \sup_{|x| \le N} | Y_{t}^{s,x} -  Y_{t}^{{s_n},x}|  
\\ \le 2 |s-{s_n}| \, \| b\|_{0}   
 +  \, \big(1_{ \{  Z_n  \le \delta \}}   +  1_{ \{  Z_n  > \delta \}} \big)  \cdot  \sup_{|x| \le N}  \int_{{s_n}}^{t} |b(r,X_r^{s,x} )  -  b(r,X_r^{{s_n},x} )  | dr 
 \\
\le   T \, 1_{ \{  Z_n  \le  \delta   \}} 
 [b]_{\beta, T}  \,  V_{s,s_n, N}^{\beta} \, \, \delta^{ \beta/2 } 
\, +  
 \, 2T \,  \| b \|_0   1_{ \{  Z_n  >  \delta   \}} +
    2 |s-{s_n}| \, \| b\|_{0}. 
\end{gather*}
Applying the expectation and using \eqref{sup2}
  we arrive at 
\begin{gather*}
 E [\sup_{s_n \le t \le T } \sup_{|x| \le N} | Y_{t}^{s,x} -  Y_{t}^{{s_n},x}| ] 
\\ \le 2 |s-{s_n}| \, \| b\|_{0}  + 
 k_0  T \, 
 [b]_{\beta, T}   \, \delta^{ \beta/2 } 
\, +  
 \, 2T \,  \| b \|_0   P  (  Z_n  >  \delta   ). 
\end{gather*} 
Now, using  \eqref{lle}, we obtain easily \eqref{de3} and this completes the proof. 
\end{proof}

In the next result we need the 
  L\'evy-It\^o formula.
 To this purpose we  recall the definition of Poisson random measure $N$: 
$
N((0,t] \times H) $ $ = \sum_{0 < s \le t} 1_{H} (\triangle L_s) 
$
 for any Borel set $H$ in ${\mathbb R}^d \setminus \{ 0 \}$; 
 $\triangle L _s$ $ = L_s - L_{s-}$ denotes the jump size of $L$
 at time $s > 0.$
 The  L\'evy-It\^o decomposition of the  given  L\'evy process $L$ on $(\Omega, {\cal F}, P)$ with  generating triplet $(\nu, Q, 0)$ (see  Section 19 in \cite{sato}
or Theorem 2.4.16 in 
\cite{A}) 
asserts that
there exists a $Q$-Wiener 
 process $B = (B_t)$ on $(\Omega, {\cal F}, P)$ independent of $N$
with covariance
matrix  $Q$ (cf. \eqref{levii})   
  such that on some almost sure event $\Omega'$ we have
\begin{equation} \label{itl}
 L_t =  A_t + B_t + C_t,
  \;\;\; t \ge 0, \;\;\; \text{where}
\end{equation}  
 \begin{equation} \label{ito1}
 A_t =
 \int_0^t \int_{\{ |x| \le 1\} } x \tilde N(ds, dx), 
 \;\;\; 
 C_t = \int_0^t \int_{\{ |x| > 1 \} } x  N(ds, dx);
 \end{equation} 
 $\tilde N$ is the compensated Poisson measure (i.e., $\tilde N (dt,  dx)$ $ = N(dt, dx)-  dt \nu(dx)$).

\begin{theorem}  \label{modi12}
Under the same assumptions of Lemma \ref{modi1a}
 consider the  process $Y = (Y^s)$ with values in 
$C({\mathbb R}^d; {G_0})$ (see (\ref{cte})). 
There exists a modification $Z = (Z^s)$ of $Y$  with c\`adl\`ag paths. 
 \end{theorem}
\begin{proof}  To prove the assertion we will  apply Corollary \ref{be1}. We already know by Lemma \ref{modi1a} that $Y$ is continuous in probability. 

In the proof we will use the fact  that 
 $\int_{\{ |x| >1 \}} |x|^{{\theta}} \nu (dx) < \infty$ for some ${\theta}  \in (0, 1) $.  This is not a restrictive according to Remark \ref{dw}. 
 We proceed in  some steps.

\hh \textit{Step I.} We establish simple moment estimates for the L\'evy process $L$, using the Ito-L\'evy decomposition \eqref{ito1}.

Using basic properties of the martingales $(A_t)$ and $(B_t)$ we obtain 
\begin{equation} \label{unoa}
 E |B_t|^2 = C_Q t, \;\;\; E |A_t|^2 =  t \int_{\{ |x| \le 1 \}} |x|^2  \nu (dx), \;\;\; t \ge 0
\end{equation} 
Now we concentrate   on the compound Poisson process $C= (C_t)$; on $\Omega$ we have
 \begin{equation*}
  |C_t|^{{\theta}} =  \Big |   \sum_{0 < s \le t }  \triangle L_s    \, 1_{ \{ |\triangle L_s| >1 
 \}} \Big|^{{\theta}}  
 \le     \sum_{0 < s \le t }  |\triangle L_s|^{{\theta}} \,   1_{ \{ |\triangle L_s| >1 
 \} },  
 \end{equation*}
 since the random sum is finite for any $\omega \in \Omega$ and ${\theta} \le 1$. Let $f_0(x) = 1_{\{|x| >1 \}}(x)\,  |x|^{{\theta}}$, $x \in {\mathbb R}^d$; using a well-know result (cf. pages 145 and 150 in \cite{K} or Section 2.3.2 in \cite{A})  we get 
\begin{equation*}
\begin{array}{l}
E \big [ \sum_{0 < s \le t }  |\triangle L_s|^{{\theta}} \,   1_{ \{ |\triangle L_s| >1 
 \} } \big] = E \big [  \int_0^t \int_{\{ |x| > 1 \} } |x|^{{\theta}}  N(ds, dx) \big]\\ = \int_{{\mathbb R}^d} f_0(x) \nu (dx) = \int_{ \{ |x| >1 \}} |x|^{{\theta}} \nu (dx) 
\end{array}
\end{equation*}
and  so
 \begin{equation} \label{jui}
 E |C_t|^{{\theta}} \le  t \int_{\{ |x| >1 \}} |x|^{{\theta}} \nu (dx) 
 = c_0 \, t,\;\;\; t \ge 0. 
 \end{equation}
 \textit{Step II.}   Let $0 \le s< s' \le T$.
  Similarly to  the proof of Lemma \ref{modi1a} in this step we 
 establish   estimates for the random variable $J_{t,x, s,s' } = $ $ | Y_{t}^{s,x} -  Y_{t}^{{s'},x}|$.

If $t \le s$
we have $J_{t,x, s,s'} =0$, $x \in {\mathbb R}^d$. If ${s} \le t \le  s' $  then,  for any $x \in {\mathbb R}^d$, on some  almost sure event  $\Omega_{s,s'}$ (independent of $t$ and $x$)  we find
$$
 | Y_{t}^{s,x} -  Y_{t}^{{s'},x}|  \le \| b\|_0  |t-{s}| \le \| b\|_0  |s -{s'}|.
$$
Let $t \ge {s'}$ and  $N \ge 1$. We have (cf. \eqref{dir2}) 
\begin{gather} \label{dir2a}
\sup_{|x| \le N} | Y_{t}^{s,x} -  Y_{t}^{{s'},x}|  \le 2 |s-{s'}| \, \| b\|_{0} +    
 \sup_{|x| \le N} \int_{{s'}}^{t} | b(r,X_r^{s,x} )  -  b(r,X_r^{{s'},x} )|
  dr.  
\end{gather}
Moreover, there exists an almost sure event $\Omega'_{s,s'} \subset \Omega_{s,s'}$ such that on $\Omega_{s,s'}'$ we have for $r \in [s', T]$
\begin{gather*}
\sup_{|x| \le N} \,  | b(r,X_r^{s,x} )  -  b(r,X_r^{{s'},x} )| =  
 \sup_{|x| \le N} \, | b(r, X_r^{{s'}, X_{{s'}}^{s,x}})    - b(r, X_r^{{s'},x})  |  
\\
\le   [b]_{\beta,T} \sup_{|x| \le N}\|  Y^{s', X_{{s'}}^{s,x}}   -  Y^{s',x}   \|_{G_0}^{\beta}.
\end{gather*}
Now we use Lemma \ref{fi} with $p \ge  32d$ 
{\it to be fixed} 
 and get,  for any $r \in [s',T]$ on $\Omega_{s,s'}'$
(cf. \eqref{deb2b} and \eqref{deb2}) 
\begin{gather*} 
\sup_{|x| \le N} \,  | b(r,X_r^{s,x} )  -  b(r,X_r^{{s'},x} )|
\\ \nonumber 
\le [b]_{\beta, T} \,  [(|x| \vee |X_{{s'}}^{s,x}|)^{\frac{2d + 1}{p}} \vee 1]^{\beta}\,  U_{s',p}^{\beta} \, \sup_{|x| \le N} \, |x-X_{{s'}}^{s,x}|^{{\beta (1- \frac{2d}{p})}}
\end{gather*}
and so
\begin{gather} \label{deb23}
\sup_{|x| \le N} \,  | b(r,X_r^{s,x} )  -  b(r,X_r^{{s'},x} )|
 \le [b]_{\beta, T} \,   V_{s, s', N,p}^{\beta} \, \sup_{|x| \le N} \, |x-X_{{s'}}^{s,x}|^{\beta (1- \frac{2d}{p})},
\\ \nonumber 
V_{s, s', N,p} = \Big[ N^{ \frac{ 2d + 1 }{p} } + (2T \| b\|_0)^{\frac{2d + 1 }{p}} + |L_{s'} - L_s|^{\frac{2d + 1 }{p}} \Big]\,  U_{s' ,p},
 \end{gather}
 Coming back to \eqref{dir2a} we find for  $t \ge {s'}$
 on  $ \Omega_{s,s'}'$ 
\begin{gather*}
\sup_{s' \le t \le T }  \sup_{|x| \le N} | Y_{t}^{s,x} -  Y_{t}^{{s'},x}|  
\le 2 |s-{s'}| \, \| b\|_{0}   
\\ +  T \, 1_{ \big \{  \sup_{|x| \le N} |X_{{s'}}^{s,x} -x| \le 
c_0 |s-s'|^{1/8}   \big \}} 
 [b]_{\beta, T}  \, V_{s, s', N,p}^{\beta} \, \sup_{|x| \le N} \, |x-X_{{s'}}^{s,x}|^{ \beta (1- \frac{2d}{p}) }.
\\+  
 \, 2T \,  \| b \|_0   1_{ \big  \{  \sup_{|x| \le N} |X_{{s'}}^{s,x} -x| > c_0 |s-s'|^{1/8}   \big \}}, 
\end{gather*}
 with $c_0 >0$ such that 
 $c_0 \rho^{1/8} - \| b\|_0\, \rho$ $ \ge \rho^{1/8}, $ for any $ \rho \in [0,T]. $ 
 We obtain  
 on   $\Omega_{s,s'}'$ 
\begin{gather} \label{dh}
 \sup_{|x| \le N} \sup_{t \in [0,T]}| Y_{t}^{s,x} -  Y_{t}^{{s'},x}|  = \sup_{|x| \le N} \| Y^{s,x} -  Y^{{s'},x}\|_{G_0} 
\\
\nonumber \le C_1 |s-{s'}| \,    
+  C_1  \, 
   V_{s, s', N,p}^{\beta} \, |s-s'|^{\frac{\beta}{8} (1- \frac{2d}{p}) }
+  
 C_1    1_{ \big \{  \sup_{|x| \le N} |X_{{s'}}^{s,x} -x| > c_0 |s-s'|^{1/8}   \big \}},
 \end{gather}
 with $C_1 = 2(T \vee 1 ) \| b\|_{\beta, T} \, c_0^{\beta}$.
Since, for any $x \in {\mathbb R}^d$, 
 $|X_{{s'}}^{s,x} -x| 
 $ $ \le |s'-s| \|b\|_0 $ $ +  |L_{{s'}} - L_s|$
  and, moreover,   
 $c_0 |s-s'|^{1/8} - \| b\|_0 |s-s'|$ 
 $\ge |s-s'|^{1/8}$, we find on $\Omega_{s,s'}'$ 
\begin{gather} \label{mis}
\sup_{|x| \le N} \| Y^{s,x} -  Y^{{s'},x}\|_{G_0} 
 \\ \nonumber
\le  C_1 \big (  |s-{s'}| \,   
+  
   V_{s, s', N,p}^{\beta} \, |s-s'|^{\frac{\beta}{8} (1- \frac{2d}{p}) }
  +    1_{ \{    |L_s'- L_s| >   |s-s'|^{1/8}   \}}  \big).
\end{gather}
Note that $C_1$
 is independent of $s$, $s'$ and $N$.

\smallskip
\hh \textit{Step III.} Using \eqref{mis} we provide an estimate for 
$
d_0 (Y^s, Y^{s'}) 
$ (cf. \eqref{metr}) when  $0 \le s < s' \le T$.

We have (see \eqref{deb23})
$$
V_{s, s', N,p}^{\beta} \le  \Big[ N^{ \frac{\beta ( 2d + 1) }{p} } + (2T \| b\|_0)^{\frac{\beta (2d + 1) }{p}} + |L_{s'} - L_s|^{\frac{\beta (2d + 1) }{p}} \Big]\,  U_{s',p }^{\beta}
$$
and so 
\begin{gather*}
d_0 (Y^s, Y^{s'}) = \sum_{N \ge 1} \frac{1}{2^N} 
\frac{\sup_{|x|\le N} \|Y^{s,x} -  Y^{{s'},x} \|_{G_0} }{1+ \sup_{|x|\le N} \|Y^{s,x} -  Y^{{s'},x} \|_{G_0} } 
\\
 \le C_1 |s-s'| + C_1 1_{ \{    |L_s'- L_s| >   |s-s'|^{1/8}   \}}
\\
+ C_1 U_{s',p }^{\beta} \, |s-s'|^{\frac{\beta}{8} (1- \frac{2d}{p})}
 \sum_{N \ge 1} \frac{1}{2^N} 
\Big[ N^{ \frac{\beta ( 2d + 1) }{p} } + (2T \| b\|_0)^{\frac{\beta (2d + 1) }{p}} + |L_{s'} - L_s|^{\frac{\beta (2d + 1) }{p}} \Big]
\\
 \le C_3  \Big ( |s-s'| +  1_{ \{    |L_s'- L_s| >   |s-s'|^{1/8}   \}}
 +
 U_{s',p }^{\beta} \, |s-s'|^{\frac{\beta}{8} (1- \frac{2d}{p})}
 \big( 1  + |L_{s'} - L_s|^{\frac{\beta (2d + 1) }{p}} \big)
 \Big),
\end{gather*}
where $C_3 = C_3 (\beta, T,  \| b\|_{\beta,T},  d,p) >0$. Recall that $p \ge 32d $ has  to be fixed.

\hh \textit {Step IV.}  
Let now  $0 \le s_1< s_2 < s_3 \le T$ and set 
$$
\rho = s_3 - s_1.
$$
We will apply Corollary \ref{be1} with $q=  8/ \beta $.   
 Let us  fix $ p \ge 32 d$ (i.e., $1- \frac{2d}{p} \ge 15/16$)  such that $\frac{8 (2d +1)}{p} < \frac{{\theta}}{4}$   
and introduce the random variable
$$
 Z = 1 + \sup_{s \in [0,T]} |L_s|^{\frac{8 (2d + 1) }{p}},
$$
Clearly we have that $|L_{s'} - L_s|^{\frac{8 (2d + 1) }{p}}
 \le 2 Z$, $0 \le s < s' \le T$. Moreover  by Remark \ref{dw}  we know that
$E [Z^4] < \infty.$  Using Step III and the previous estimates  we will check condition \eqref{bez11}.  
In the sequel we denote by $C_k$ or $c_k$ positive constants which may depend on $\beta, T,  \| b\|_{\beta,T}, {\theta}$ and  $d$ 
but are independent of $s_1$, $s_2$ and $s_3$.
 We have 
\begin{gather*}
\Gamma = E \Big [ \Big ( d_0  (Y^{s_1} , Y^{{s_2}} )  \cdot 
  d_0 ( Y^{s_2} ,  Y^{{s_3}} ) \Big)^{8/\beta} \Big]
  \\ \le 
C_4 E \Big[  \Big ( |s_3 -s_1|^{8/\beta} +  1_{ \{    |L_{s_2}- L_{s_1}| >   |s_2 -s_1|^{1/8}   \}}  +
 Z \, U_{s_2,p }^{8} \, |s_3-s_1|^{1- \frac{2d}{p} }
  \Big) \, 
  \\
 \cdot \, 
\Big ( |s_3 -s_1|^{8/\beta} +  1_{ \{    |L_{s_3}- L_{s_2}| >   |s_3 -s_2|^{1/8}   \}}
 +
Z\,  U_{s_3,p }^{8} \, |s_3-s_1|^{1- \frac{2d}{p} }
  \Big) \Big]. 
\end{gather*}
We denote by $c_2 \ge 1$ a constant such that 
 $t^{8/\beta} \le c_2 t^{1 - \frac{2d}{p}}  $, $t \in [0,T]$. 
We obtain $(\rho = s_3 - s_1)$
\begin{gather*}
\Gamma \le 
c_2^2 \, C_4 E \Big[  \Big ( \rho^{1- \frac{2d}{p}} +  1_{ \{    |L_{s_2}- L_{s_1}| >    |s_2 -s_1|^{1/8}   \}}  +
 Z \, U_{s_2,p }^{8} \, \rho^{1- \frac{2d}{p} }
  \Big) \, 
  \\
 \cdot \, 
\Big ( \rho^{1- \frac{2d}{p}} +  1_{ \{    |L_{s_3}- L_{s_2}| >   |s_3 -s_2|^{1/8}   \}}
 +
Z\,  U_{s_3,p }^{8} \, \rho^{1- \frac{2d}{p} }
  \Big) \Big] 
  \\ 
  \nonumber \le C_5 ( \Gamma_1 + \Gamma_2 + \Gamma_3 + \Gamma_4),
\end{gather*}
where $\Gamma_1 =  E [ 1_{ \{    |L_{s_2}- L_{s_1}| >    |s_2 -s_1|^{1/8}   \}}  \cdot  1_{ \{    |L_{s_3}- L_{s_2}| >    |s_3 -s_2|^{1/8}   \}} ]$,
\begin{equation*}
\Gamma_2 =  \rho^{1- \frac{2d}{p}} [ P ( |L_{s_3}- L_{s_2}| >    |s_3 -s_2|^{1/8} +  P ( |L_{s_2}- L_{s_1}| >    |s_2 -s_1|^{1/8} )], 
\end{equation*}
$$
\Gamma_3 = \rho^{1- \frac{2d}{p}}  E [1_{ \{    |L_{s_3}- L_{s_2}| >   |s_3 -s_2|^{1/8}   \}}
 Z\,  U_{s_2,p }^{8}  + 
1_{ \{    |L_{s_2}- L_{s_1}| >   |s_2 -s_1|^{1/8}   \}}
Z\,  U_{s_3,p }^{8} ], 
$$
\begin{gather*}
 \Gamma_4  = \rho^{2 (1- \frac{2d}{p})}
    + \rho^{2 (1- \frac{2d}{p})} E [Z  U_{s_2,p }^{8} + 
     Z U_{s_3,p }^{8} + Z^2 U_{s_2,p }^{8} U_{s_3,p }^{8} ].
\end{gather*}
It is not difficult to treat $\Gamma_4$. Indeed we can use 
the Cauchy-Schwarz inequality and 
\begin{equation} \label{stimo} 
 \sup_{s \in [0,T]} E [U_{s,p}^{p}] = k'  < \infty 
\end{equation} (see \eqref{ri4}) in  order to  control the expectation in $\Gamma_4$. For instance, we have
\begin{equation} \label{nizza}
 E [Z^2 U_{s_2,p }^{8} U_{s_3,p }^{8}] \le
E [Z^4]^{1/2}  (\sup_{s \in [0,T]} E [U_{s,p }^{32}])^{1/2}
< \infty, \end{equation}
since $E [Z^4] < \infty$ and $ p \ge 32 d$. We obtain
\begin{equation} \label{one}
 \Gamma_4 \le C_6 \rho^{2 (1- \frac{2d}{p})} =
  C_6 |s_3 -s_1|^{2 (1- \frac{2d}{p})} \le C_6 |s_3 -s_1|^{30/16}.
\end{equation} 
 To estimates the other terms we need to control $ P ( |L_{s}| >    |s|^{1/8} ) $, $s \ge 0.$
To this purpose we use Step I.
We have 
\begin{gather*}
 P ( |L_{s}| >    s^{1/8} ) \le
  P ( |B_{s}| >    s^{1/8}/3 ) + 
   P ( |A_{s}| >    s^{1/8}/3 )
 +  P ( |C_{s}| >    s^{1/8}/3 ).
\end{gather*}
By Chebychev inequality we get for $s \ge 0$
\begin{gather} \label{vaii}
 P ( |L_{s}| >    s^{1/8} ) \le
 \frac{9}{s^{1/4}} E [|B_{s}|^2 + |A_s|^2]   + 
    \frac{3^{{\theta}}}{s^{{\theta}/8}} E [|C_s|^{{\theta}}] 
\le c_3  (s^{3/4} + s^{1- \frac{{\theta}}{8}}).    
\end{gather}
Using \eqref{vaii} and \eqref{stimo} we can estimate $\Gamma_2$
and $\Gamma_3$. For instance, since the increments of $L$  are independent and  stationary, we find   
\begin{gather*}
\Gamma_2 \le  \rho^{1- \frac{2d}{p}} [ P ( |L_{s_3 - s_2}| >    |s_3 -s_2|^{1/8}) +  P ( |L_{s_2 -s_1}| >    |s_2 -s_1|^{1/8} )] \\ \le 2 c_3 \rho^{1- \frac{2d}{p}} (\rho^{3/4} + \rho^{1- \frac{{\theta}}{8}}). 
\end{gather*}
We can proceed similary for $\Gamma_3$ (see also \eqref{nizza}):  
\begin{gather*}
\Gamma_3 \le  \rho^{1- \frac{2d}{p}} (E [Z^4])^{1/4}
(\sup_{s \in [0,T]} E [U_{s,p }^{32}])^{1/4} 
 \, \Big [ (P ( |L_{s_3 - s_2}| >    |s_3 -s_2|^{1/8})^{1/2} 
\\ +  (P ( |L_{s_2 -s_1}| >    |s_2 -s_1|^{1/8} ))^{1/2} \Big] 
\le C_8  \rho^{1- \frac{2d}{p}} (\rho^{3/8} + \rho^{\frac{1}{2} (1- \frac{{\theta}}{8}) }). 
\end{gather*}
Note that $(1- \frac{2d}{p}) + {3/8} >5/4 $  and  
 $(1- \frac{2d}{p})+ \frac{1}{2} (1- \frac{{\theta}}{8}) >5/4$.
 We get 
\begin{equation} \label{two}
 \Gamma_2 + \Gamma_3 \le C_9 \rho^{\frac{5}{4}} =
  C_9 |s_3 -s_1|^{5/4}.
\end{equation} 
Finally we consider  
\begin{gather} \label{three}
 \Gamma_1  \le 
 (P ( |L_{s_3 - s_2}| >    |s_3 -s_2|^{1/8}) \, 
 \cdot \, (P ( |L_{s_2 - s_1}| >    |s_2 -s_1|^{1/8})
\\ \nonumber 
\le  2c_3  (\rho^{3/2} + \rho^{2(1- \frac{{\theta}}{8})}) 
 \le  c_4 |s_3- s_1|^{3/2}.
\end{gather} 
Collecting together estimates \eqref{one}, \eqref{two} and \eqref{three} we arrive at 
$$
E \Big [ \Big ( d_0  (Y^{s_1} , Y^{{s_2}} )  \cdot 
  d_0 ( Y^{s_2} ,  Y^{{s_3}} ) \Big)^{8/\beta} \Big]
\le C_{0} |s_3- s_1|^{5/4}
$$
and this finishes the proof.
\end{proof}
Taking into account  Theorem  \ref{modi12} 
 and using  the projections $\pi_x$ (see \eqref{pro2}),
in the sequel we write, for $x \in {\mathbb R}^d$, $s,t \in [0,T]$, 
\begin{equation} \label{dobp}
Z^s = (Z^{s,x})_{x \in {\mathbb R}^d}, \;\;\; \text{with} \;  
\pi_x (Z^s)= Z^{s,x} \in G_0.
\end{equation}
Recall that on some almost sure event $\Omega_s$,  $Y^{s,x} 
 = Z^{s,x}$, $s \in [0,T]$, $x \in \R^d$
(cf. \eqref{cte}).  
\begin{lemma}
\label{le1} Under the same assumptions of Lemma  \ref{modi1a}
 consider the c\`adl\`ag process  $Z$ with values in $C(\R^d; G_0)$ of Theorem \ref{modi12}. The following statements hold:
\hh (i) There exists an almost sure event $\Omega_1$
 (independent of $s,t$ and $x$)
 such that for any $\omega \in {\Omega}_1$, we have
 that $t \mapsto L_t (\omega)$ is c\`adl\`ag, $L_0(\omega)=0$ and 
 $s \mapsto Z^s(\omega)$ is c\`adl\`ag; further, for any $\omega \in \Omega_1$, 
$$
  \;\; Z^{s,x}_t(\omega) = x + \int_s^{t} b(r, Z^{s,x}_r(\omega)  + L_{r}(\omega) - L_s(\omega))dr, \; s,t \in [0,T], \; s \le t, \,x \in {\mathbb R}^d.
$$
Moreover,  for $s \le t $, the r.v. $Z^{s,x}_t$ is ${\cal F}_{s,t}^L$-measurable (if $t \le s $, $Z^{s,x}_t =x$).

\hh (ii) There exists an almost sure event ${\Omega}_2$ and  a  ${\cal B}([0,T]) \times {\cal F}$-measurable function  $V_{n}: [0,T] \times \Omega \to [0, \infty]$, such that $\int_0^T V_{n}(s,\omega)ds < \infty$, for any integer $n >2d$, $\omega \in {\Omega}_2$, and, further, the following inequality holds on ${\Omega}_2$ 
\begin{equation} 
\label{tor34}
\sup_{t \in [0,T]} |Z_t^{s,x} - Z_t^{s,y}|  
\le  |x-y|^{\frac{n - 2d}{n}} [(|x| \vee |y|)^{\frac{2d + 1}{n}} \vee 1] \, V_{n}(s, \cdot),  \; x,y \in {\mathbb R}^d, \, s \in [0,T].
\end{equation} 
(iii) There exists an almost sure event ${\Omega}_3$ such 
that for any $\omega \in \Omega_3$ we have
\begin{equation}
 \label{flo2}
 Z_t^{s,x} (\omega) + L_u(\omega) - L_s(\omega)
  = Z_t^{u, \, Z_u^{s,x}(\omega) + L_u(\omega) - L_s(\omega)}\, (\omega), \,
\end{equation} 
for any $s,u,t \in [0,T]$, $0 \le s < u \le T$, $x \in {\mathbb R}^d$.
\end{lemma}
\begin {proof}

\hh {\bf (i)}
 On some almost sure event  $ \Omega_{s}'$ (independent of $t$  and $x$) we know that $(Y_t^{s,x})$ verifies the SDE \eqref{yrye}
 for any $x \in {\mathbb R}^d$ and $t \in [s,T]$. Moreover  
 $Y_t^{s,x} = x$, $t<s$. 

On the other hand  on some almost sure event  $\Omega_s$ we have $Y^{s,x}$ $ = \pi_x(Y^{s}) = \pi_x(Z^{s})$, for any $x \in {\mathbb R}^d$, see \eqref{dobp}.  Using   $(Z^s)$,
we can rewrite  \eqref{yrye} on the event  ${\Omega}_1 = \bigcap_{r \in \Q \cap [0,T]} (\Omega_{r}' \cap \Omega_r)$ as follows:
  \begin {equation} \label{d113}
[\pi_x(Z^{s})]_t = x + \int_s^t b(r,  [\pi_x(Z^{s})]_r + (L_r - L_s)) dr, 
\end{equation}
for any  $s \in \Q \cap [0,T]$, $t \in [s,T]$, $x \in \R^d.$ 

Note  that by Theorem \ref{modi12} for any $\omega \in \Omega$ and any sequence $s_n \to s^+$ we have  $d_0 ( Z^{s}  (\omega),  Z^{s_n} (\omega))$ $\to 0$ as $n \to \infty$. Take now $s \in [0,T)$ and let $(s_n) \subset \Q \cap [0,T]$ be a sequence monotonically decreasing to $s$. By the dominated convergence theorem and the right-continuity of $L$ we have on $\Omega_1$, for any $t > s$,  $x \in \R^d,$  
\begin{gather*}
[\pi_x(Z^{s})]_t = \lim_{n \to \infty} [\pi_x(Z^{s_n})]_t
= 
x + \lim_{n \to \infty}  \int_s^t 1_{\{ r > s_n \}} \, b(r,  [\pi_x(Z^{s_n})]_r + (L_r - L_{s_n})) dr
\\ 
= x + \int_s^t b(r,  [\pi_x(Z^{s})]_r + (L_r - L_s)) dr
\end{gather*}
and we get the assertion.

\hh {\bf (ii)} Since on  
 $\Omega$ we have $Y^{s,x}$ $ = \pi_x(Y^{s})$ we obtain by \eqref{ciao221} and \eqref{yy12},   for any $p \ge 2$,  
\begin{gather}
\label{cia} 
\sup_{s \in [0,T]}  \, E[\sup_{ s \le t \le T} |X_t^{s,x} - X_t^{s,y}|^p] 
= \sup_{s \in [0,T]}  \, E[\sup_{ 0 \le t \le T} |Y_t^{s,x} - Y_t^{s,y}|^p] 
\\
\nonumber 
= \sup_{s \in [0,T]}  \, E[\| \pi_x (Z^{s})  -  \pi_y (Z^{s}) \|^p_{G_0}] 
\le C(T) \,
 |x-y|^p, \;\;\;  x,\, y \in {\mathbb R}^d.
\end{gather}  
Let $s \in [0,T]$   and consider the random field $( \pi_x (Z^{s}))_{x \in {\mathbb R}^d}$ with values in ${G_0}$.    
Applying Theorem \ref{grr} with $\psi (x,\omega) = \pi_x (Z^{s})(\omega)$ we obtain from \eqref{cia} for $p > 2d$ similarly  to \eqref{est23}:
there exists a   $V_{p}(s, \omega ) \in [0, \infty]$ such that, for any  $\omega \in \!\Omega$, $x, y \in {\mathbb R}^d$, 
$s \in [0,T]$,
\begin{equation}
\label{est231}
\|  \pi_x (Z^{s})(\omega)   -  \pi_y(Z^{s}) (\omega)  \|_{G_0} \le [(|x| \vee |y|)^{\frac{2d + 1}{p}} \vee 1] \,   V_{p}(s,\omega) |x-y|^{1- 2d/p}, 
\end{equation} 
with 
$V_{p} (s,\omega) =  \Big(\int_{{\mathbb R}^d} \int_{{\mathbb R}^d} \Big( \frac{\| \pi_x(Z^{s}) (\omega)     -  \pi_y(Z^{s}) (\omega)    \|_{G_0}}{|x-y|}\Big)^p \;    f(x) f(y) dxdy\Big)^{1/p},$ $ \omega \in \Omega,$
$s \in [0,T]$ ($f$ is defined in Theorem \ref{grr}). Since the map: $(s,x, \omega)$ $ 
\mapsto \pi_x(Z^{s}) (\omega)$ is ${\cal B}([0,T] \times {\mathbb R}^{d} ) \times {\cal F}$-measurable with values in $G_0$, it follows that the real map: 
$$
(s,x,y, \omega) \mapsto {\| \pi_x(Z^{s}) (\omega)    -  \pi_y(Z^{s}) (\omega)    \|_{G_0}}{|x-y|^{-1}} \, 1_{\{ x \not = y\}}
$$ is ${\cal B}([0,T] \times {\mathbb R}^{2d} ) \times {\cal F}$-measurable. By the Fubini theorem
 we deduce  that also $V_p : [0,T] \times \Omega \to [0, \infty]$ is  ${\cal B}([0,T]) \times {\cal F}$-measurable.     
 Hence we can consider the  
 random variable $\omega \mapsto \int_0^T V_{p}(s,\omega)  ds $ (with values in $[0, \infty]$).
Since, with  the same constant $C(T)$ appearing   in \eqref{ciao221}, 
\begin{equation} \label{ri41}
\sup_{s \in [0,T]} \, E [ |V_{p}(s, \cdot)|^p] \le C(d)\cdot C(T),
\end{equation}  
  we  find
 $ E \Big[ \Big ( \int_0^T V_{p}(s, \cdot)  ds \Big)^p \Big] \le T^{{p-1}}
    \int_0^T E [ (V_{p}(s, \cdot))^p ] ds  \le T^{{2p-1}} c(d) C(T)
 < \infty.$
 It follows that, for any $p >2d ,$
  there exists an almost sure event $\Omega_p$ such that  
\begin{equation}
\label{dsw}
 \int_0^T V_{p}(s,\omega)  ds < \infty, \;\;\; \omega \in \Omega_p.
\end{equation} 
 Let $p=n$.
  We find, for any $n > 2d $, $\int_0^T V_{n}(s,\omega)  ds < \infty,$ when $ \omega$ $ \in  {\Omega}_2 $ $ = \bigcap_{n > 2d}\, \Omega_n$.

 Writing \eqref{est231} for $\omega \in \Omega_2$ and $n > 2d$
 we find the assertion.

\hh \hh {\bf (iii)} First note that the statement of Lemma \ref{fll}
can be rewritten in term of the process $Y^{s,x}$ (see \eqref{yy12}) as follows: for any $0 \le s < u \le T$
there exists an almost sure event ${\Omega}_{s, u}$ (independent of $t$ and $x$) such 
that, for any $\omega \in \Omega_{s,u}$, we have
\begin{equation}
 \label{flora}
 Y_t^{s,x} (\omega) + L_u(\omega) - L_s(\omega)
  = Y_t^{u, \, Y_u^{s,x}(\omega) + L_u(\omega) - L_s(\omega)}\, (\omega),  
  \; \text{for   $t \in [u,T],  x \in {\mathbb R}^d.$}
\end{equation} 
Since $(Z^{s})$ is a modification of $(Y^{s})$ (see Theorem \ref{modi12})  we know that 
on some almost sure event $\Omega_{s,u}'' \subset \Omega_{s,u}$ identity \eqref{flora} holds when $(Y^{s,x})$ is replaced by $(Z^{s,x})$.

 Let us fix $u \in (0,T]$. We know that \eqref{flora} holds for  $(Z^{s,x})$ when $t \in [u,T]$, $x \in \R^d$ and $s \in [0,u) \cap \Q$ if $\omega \in 
\Omega_{u} = \cap _{s \in  [0,u) \cap \Q} (\Omega_{s,u}'' \cap \Omega_1)$. 
 Using that $(Z^{s,x})$ with values in $G_0$ is in particular right-continuous in $s$, uniformly in $x$, when $x$ varies in compact sets of $\R^d$, it easy to check that \eqref{flo2}
 holds, for any  $0 \le s < u \le T$, $x \in {\mathbb R}^d$,
$ t \in [u,T]$, when $\omega \in  \Omega_u$. 

Let us define ${\Omega}_3 = \bigcap_{u \in \Q \cap [0,T)} \Omega_{u} $; fix any $s, u_0 \in [0,T]$, $x \in \R^d$, with $0 \le s < u_0 \le T$; we consider $\omega \in \Omega_3$ and prove that \eqref{flo2} holds for any $t \in [u_0,T]$.
 
 If $t =u_0$ the  assertion holds.
 Let us suppose that $t \in (u_0, T]$.
 we can find a sequence $(u_j) \in (u_0,t) \cap \Q$ such that $u_j \to u_0^+$. Since for any $j \ge 1$ we have 
\begin{equation}
 \label{flo2r}
 Z_t^{s,x} (\omega) + L_{u_j}(\omega) - L_s(\omega)
  = Z_t^{u_j, \, Z_{u_j}^{s,x}(\omega) + L_{u_j}(\omega) - L_s(\omega)}\, (\omega), \,
\end{equation}  
we can pass to the limit as $j \to \infty $ in both sides of the previous formula (taking also into account that $Z_{u_j}^{s,x}(\omega) + L_{u_j}(\omega) - L_s(\omega)$ belongs to a 
compact set $K_{x,s,\omega} \subset \R^d$ for any $j \ge 1$) 
and find that \eqref{flo2r} holds when $u_j$ is replaced by $u_0$. 
 The proof of \eqref{flo2} is complete. 
\end{proof}

\section{A Davie's type uniqueness  result}

Assertion (v) of the next theorem gives a Davie's type uniqueness result for SDE \eqref{SDE}. The other assertions  collect results of Section 4 (see in particular Theorem \ref{modi12} and Lemma \ref{le1}). These  are  used to prove the  uniqueness property (v). We refer to 
Corollaries \ref{unb} and \ref{cons} for the case when $b(t, \cdot)$ is only locally H\"older continuous.

\smallskip
{\sl We stress that all the next statements (i)-(v) hold when $\omega $
 belongs to an almost sure event $\Omega'$ (independent of $s$, $t \in [0,T]$ and $x \in {\mathbb R}^d$).}
\begin{theorem} \label{d32} Let us consider the SDE \eqref{SDE} 
with $b \in L^{\infty}(0,T; C_b^{0, \beta}(\R^d; \R^d))$, $\beta \in (0,1],$   
and suppose that $L$ and $b$ 
 satisfy
 Hypotheses \ref{primo} 
  and  \ref{zero}. 
 Then there exists  a  function  $\phi (s,t,x, \omega)$, 
\begin{equation} \label{nuova}
\phi : [0,T] \times [0,T] \times {\mathbb R}^d \times \Omega \to {\mathbb R}^d,
\end{equation}
which is ${\cal B}([0,T] \times [0,T] \times {\mathbb R}^d) \times   {\cal F} $-measurable
and 
such that $\big ( \phi(s,t,x , \cdot ) \big)_{t \in [0,T]}$
is a strong solution of \eqref{SDE} starting from $x$ at time $s.$
 Moreover, there exists an almost sure event $\Omega'$ 
  such that 
{\sl the following assertions hold for any $\omega \in \Omega'$.}
 
\hh \hh 
(i) For any $x \in {\mathbb R}^d$,  the mapping: 
 $s \mapsto \phi (s,t,x, \omega)$ is c\`adl\`ag on $[0,T]$  (uniformly in $t$ and $x$), i.e., 
  let $s \in (0,T)$ and consider sequences $(s_k) $ and $(r_n)$ such that  $s_k \to s^-$ and $r_n \to s^+$; we have, for any $M>0$,
\begin{gather} \label{sd}
 \lim_{n \to \infty} \sup_{|x| \le M}\sup_{t \in [0,T]} |\phi (r_n,t,x, \omega) - \phi (s,t,x, \omega) | =0,
\\ \nonumber  
\lim_{k \to \infty} \sup_{|x| \le M}\sup_{t \in [0,T]} |\phi (s_k,t,x, \omega) - \phi (s-,t,x, \omega) | =0
\end{gather}
(similar conditions hold   when $s =0$ and  $s=T$). 
\hh (ii) For any $x \in {\mathbb R}^d$, $s \in [0,T]$,  
 $ \phi (s,t,x, \omega) =x $ if $\, 0 \le t \le s$, and 
\begin{equation} \label{d566}
 \phi (s,t,x, \omega) = x + \int_{s}^{t}b\left(r,  \phi (s,r,x, \omega) \right) dr + L_t(\omega) - L_s(\omega),\;\;\; t \in [s,T].
\end{equation} 

\hh (iii) For any $s \in [0,T]$, the function $x \mapsto \phi (s,t,x, \omega)$ is continuous in $x$ uniformly in $t$. Moreover, 
 for any  integer $n > 2d$, there exists a ${\cal B}([0,T])\times {\cal F}$-measurable function  $V_n : [0,T] \times \Omega  \to [0, \infty]$ such that $\int_0^T V_n(s , \omega) ds < \infty$ and 
\begin{gather} 
\label{toro}
\sup_{t \in [0,T]} |\phi (s,t,x, \omega)  - \phi (s,t,y, \omega)| 
\\ \nonumber 
\le V_n(s, \omega) \,  |x-y|^{\frac{n - 2d}{n}} [(|x| \vee |y|)^{\frac{2d + 1}{n}} \vee 1],\;\;\; x,y 
\in {\mathbb R}^d, \, n > 2d , \; s \in [0,T].  
\end{gather} 
(iv) For any $0 \le s <  r \le t \le T$, $x \in {\mathbb R}^d$, we have 
\begin{equation} \label{fl2}
 \phi (s ,t,x, \omega) = \phi (r ,t, \phi (s ,r,x, \omega), \omega).
\end{equation} 
(v) 
  Let $s_0 \in [0,T)$, $\tau = \tau(\omega) \in (s_0, T]$  and  $x \in {\mathbb R}^d$. If  a measurable  function  
 $g:  [s_0,\tau ) \to  {\mathbb R}^d$ solves the integral equation
\begin{equation} \label{integral} 
g(t) =  x + \int_{s_0}^{t}b\left(r, g(r)\right)  dr 
 +  L_{t} (\omega) - L_{s_0}(\omega),\;\;\; t \in [s_0,\tau), 
\end{equation}
then we have $g(r) = \phi(s_0,r,x, \omega)$, for $r \in [s_0,\tau)$. 
\end{theorem}
\begin{proof}  Let us consider the process $Z= (Z^s)_{s \in [0,T]}$ of Theorem \ref{modi12} with values in $C({\mathbb R}^d; {G_0})$. 
 Recall the notation $Z^{s,x}_{t} = \pi_x(Z^s)(t)$ (see \eqref{pro2}).
We define for $\omega \in \Omega, \; s,t \in [0,T], \; x \in {\mathbb R}^d$:   
\begin{equation} \label{d114}
\phi (s,t,x,\omega) = Z^{s,x}_{t}(\omega) + L_{t}(\omega) - L_s(\omega),\;\;\; \text{if}  \,\, s \le t,
\end{equation}
and $\phi (s,t,x,\omega) =  x$ if  $s > t$.
The   fact that, for any $0 \le s< t \le T$, $x \in {\mathbb R}^d$,  the random variable $\phi(s,t,x, \cdot)$ is ${\cal F}_{s,t}^L$-measurable follow from Theorem \ref{modi12} and (i) in Lemma \ref{le1}.
 We also define 
$$
\Omega' = \Omega_1 \cap \Omega_2 \cap \Omega_3,
$$
where the almost sure events $\Omega_k$, $k=1,2,3$, are considered in Lemma \ref{le1}. 

\hh Assertions  {\bf (i), (ii), (iii), (iv) } follow directly from Theorem \ref{modi12} and Lemma \ref{le1}.  
More precisely, (i) and (ii) follow from the first assertion of Lemma \ref{le1} since $(Z^s)$ takes  values in $C({\mathbb R}^d; G_0)$ with c\`adl\`ag paths.
 Assertions (iii)  and (iv) follow  respectively from  the second and third assertion of Lemma \ref{le1}.  

\hh \hh {\bf (v)}
Let $\omega \in \Omega'$ be fixed  and let    $g : [s_0,\tau[ \to {\mathbb R}^d$ be a  solution to the integral equation \eqref{integral} corresponding to $\omega$.
 Let us fix $t \in (s_0, \tau)$.

We introduce  an  auxiliary function $f: [s_0, t] \to {\mathbb R}^d$ which is similar  to  the one used in proof  of Theorem 3.1 in \cite{Sh},
 \begin{equation} \label{fun2}
f(s) = \phi(s,t, g(s), \omega),\;\;\; s \in [s_0,t].
\end{equation}
We will  show
 that $f$ is 
constant on $[s_0,t]$. Once this is proved   we can deduce that 
 $f(t) = f(s_0)$ and so  we find
 $ g(t) $ $= \phi(s_0,t, x, \omega)$
 which  shows the assertion  since $t$ is arbitrary.
 In the  sequel  we proceed in three  steps.

\hh \textit{Step I. } We  establish some  estimates  for 
 $| g(r)  -  \phi(u,r, g(u), \omega) | $
when $s_0 \le u \le r \le t$.

Since
$$
g(r) = x + \int_{s_0}^u b(p, g(p)) dp + (L_u(\omega) - L_{s_0}(\omega)) +   \int_u^r b(p, g(p)) dp + (L_r(\omega) - L_u
(\omega)),$$ 
we obtain
\begin{gather*}
\nonumber | g(r)  -  \phi(u,r, g(u), \omega) |  \le \Big | g(u) + \int_u^r b(p, g(p))dp  + ( L_r(\omega) -L_u(\omega)  ) - g(u) \\  -  \int_u^r b(p, 
 \phi(u,p, g(u), \omega))dp  - ( L_r(\omega) -L_u(\omega) ) \Big|
\\ \nonumber
\le    \int_u^r | b(p,g(p)) -  b(p, 
 \phi(u,p, g(u), \omega))| dp \le 2 \| b\|_0 \, |r-u|.
\end{gather*}
Now using the H\"older continuity
 of $b$:
\begin{gather}  \label{gg}
| g(r)  -  \phi(u,r, g(u), \omega) |    
\le    \int_u^r | b(p,g(p)) -  b(p, 
 \phi(u,p, g(u), \omega))| dp  
\\ \nonumber \le [b]_{\beta, T}  \int_u^r  | g(p) -   
 \phi(u,p, g(u), \omega)|^{\beta}dp 
\\ \nonumber
\le  \, 
 (2 \| b\|_0)^{\beta} \, [b]_{\beta,T} 
  \, \int_u^r 
 |p-u|^{\beta}  \, dp
 \le (2 \| b\|_0)^{\beta} \, [b]_{\beta,T}  \,   |r-u|^{1+ \beta}.
\end{gather}
\textit{II Step.} We prove that $f$ defined in \eqref{fun2} is continuous on $[s_0,t]$.

\smallskip
We first show that it is right-continuous on $[s_0,t)$.
Let us fix $s \in [s_0,t)$  and consider a sequence  $(s_n)$ such that  $s_n \to s^+$. We prove that $f(s_n) \to f(s)$
as $n \to \infty.$
 Note that 
$|g(r)| \le M_0$, $r \in [s_0,\tau)$, where $M_0 =|x| + T \| b\|_0 + C(\omega)$. We have
 \begin{gather*}
|f(s_n) - f(s)| \le |\phi (s_n,t, g(s_n), \omega) - \phi (s,t,g(s_n), \omega)  |\\  +
|\phi (s,t, g(s_n), \omega) - \phi (s,t,g(s), \omega) |
  \le  J_n + I_n, 
\end{gather*}  
 {where} $I_n =  |\phi (s,t, g(s_n), \omega) - \phi (s,t,g(s), \omega)|$ and 
 $$J_n =  \sup_{|x| \le M_0}\sup_{t \in [0,T]} |\phi (s_n,t,x, \omega) - \phi (s,t,x, \omega) |.
 $$
Since $g(s_n ) \to g(s)$ by the right continuity of $g$ we obtain that $\lim_{n \to \infty} I_n =0$ thanks to \eqref{toro}. Moreover  $\lim_{n \to \infty} J_n =0$  thanks to \eqref{sd}.

Let us  show that $f$  is left-continuous on $(s_0,t]$.
We fix $s \in (s_0,t]$  and consider a sequence  $(s_k) \subset (s_0,s)$ such that  $s_k \to s$. We prove that $f(s_k) \to f(s)$
as $k \to \infty.$
 Using the flow property (iv) we find  
\begin{gather*}
|f(s_k) - f(s)| =  |\phi (s_k,t, g(s_k), \omega) - \phi (s,t,g(s), \omega)  |\\  
= |\phi (s,t, \phi (s_k,s, g(s_k), \omega) , \omega) - \phi (s,t,g(s), \omega)| .
\end{gather*}
By I Step  we know that  
\begin{equation} \label{dani}
|\phi (s_k,s, g(s_k), \omega)  -   g(s)| 
\le  2 \| b\|_0 |s_k -s|^{}
\end{equation} which tends to $0$
 as $k \to \infty$.  Using \eqref{dani} and the continuity property (iii)
we obtain  the claim since
$$
 \lim_{k \to \infty} |\phi (s,t, \phi (s_k,s, g(s_k), \omega), \omega) - \phi (s,t,g(s), \omega) | =0.
$$
\textit{III Step.} We prove that $f$ is constant on $[s_0,t]$.

We will use the following well known lemma (see, for instance, pages 239-240 in \cite{Y}): 
 {\sl Let $S$ be a real Banach space and consider a continuous mapping $F: [a,b] \subset \R \to S$, $b>a$.  Suppose that for any $h \in (a,b]$ there exists the left derivative 
\begin{equation}
 \label{left}
 \frac{d^- F}{dh}(h) = \lim_{h' \to h^-} \frac{F(h' ) - F(h) }{h'-h} 
\end{equation} 
and this derivative is identically zero   on $(a,b]$. Then $F$ is constant.
}

Note that by considering continuous linear functionals on $S$ one  may reduce  the proof of the  lemma to the one of a 
 real analysis  result.

To apply  the previous lemma with $[s_0,t] = [a,b]$ we first 
extend our function $f$ to $[s_0, \infty)$ by setting 
 $f(r) = f(t)$ for $r \ge t$. 
 Then set $S = L^1([0,t] ; {\mathbb R}^d)$ and define $F : [s_0,t] \to S$  as follows:
$
 F(h) = f(\cdot + h) \in S$, $  h \in [s_0,t],
$ i.e., 
$
F(h)(r) = f(r+h), \; 
r \in [0,t]. 
$

  If we prove that the mapping $F$ is constant then we deduce
(taking $h=s_0$ and $h=t$) that $f(s_0 + \cdot) = f(t + \cdot) = f(t)$ in $S$. However, since $f$ is continuous this  implies that $f$ is constant and finishes the proof.

The continuity of $F$, i.e., for any $h \in [s_0,t]$, we have
$$
\lim_{h' \to h} \| F(h) - F(h') \|_S = \lim_{h' \to h} \int_{0}^t |f(r+h) - f(r+h')|dr =0,
$$
 is clear, using the continuity of $f$. Let us prove that the left derivative of $F$ is identically  zero on $(s_0,t]$.
 
 Using the flow property (iv)  we find, for  $h, h' \in [s_0, t]$, $h' < h$ 
  and $0 \le r \le t -h$,
\begin{gather} \label{droo}
|f(r + h) - f(r+h')|  
\\ \nonumber = | \phi(r+h,t, g(r+h), \omega)  -  \phi(r+h,t,  \phi(r+h',r+h, g(r+ h'), \omega) , \omega) |.
\end{gather}
Using \eqref{droo} and changing variable we obtain  (recall that $f(r) = f(t)$, $r \ge t$)  
 \begin{gather} 
\label{torta}
\int_{0}^t |f(r+h ) - f(r+h')| dr  
\\ 
\nonumber = 
\int_{ 0}^{t-h} 
| \phi(r+h,t, g(r+h), \omega)  -  \phi(r+h,t,  \phi(r+h',r+h, g(r+ h'), \omega) , \omega) |dr
\\ \nonumber
+ \int_{t-h}^{t-h'} |f(t ) - f(r+h')| dr 
\\ \nonumber
= \int_{  h}^{t} 
  | \phi(p,t, g(p), \omega)  -  \phi(p,t,  \phi(p+h'-h,p, g(p+ h'-h), \omega) , \omega) |dp 
 \\ \nonumber
+\int_{t-h}^{t-h'} |f(t ) - f(r+h')| dr. 
\end{gather}
In order to estimate $\| F(h) - F(h') \|_S $
let us denote by $\lambda_f$ the modulus of continuity of $f$. Since in the last integral $t- h +h' \le  r+h' \le t$ we have the estimate
$$
\int_{t-h}^{t-h'} |f(t ) - f(r+h')| dr \le |h-h'|\, \lambda_f(|h-h'|)
$$
 and $\lim_{r \to 0^+} \lambda_f(r) =0$.
Taking into account that 
there exists a  constant $N_0 = N_0(x,T, \| b\|_0, \omega) \ge 1$ such that  
$$
 |g(r) | +  |\phi(r,u, g(r), \omega)| \le N_0, \;\;\; s_0 \le r \le u \le T, 
$$
we find for $p \in [h,t]$, $n > 2d$
 (see \eqref{toro} and \eqref{gg})
\begin{gather*} | \phi(p,t, g(p), \omega)  -  \phi(p,t,  \phi(p+h'-h,p, g(p+ h'-h), \omega) , \omega) | 
 \\ \le V_n(p, \omega) \,  | 
 g(p) -  \phi(p+h'-h,p, g(p+ h'-h), \omega) |^{\frac{n - 2d}{n}}  N_0^{\frac{2d + 1}{n}}
\\ \le (2 \| b\|_0)^{\beta ({\frac{n - 2d}{n}})} \, [b]_{\beta,T}^{{\frac{n - 2d}{n}}}  
\, \,  V_n(p, \omega) \,   | h'-h|^{(1+ \beta) ({\frac{n - 2d}{n}})}\,    N_0^{\frac{2d + 1}{n}}.  
\end{gather*}
 Recall that $V_n(p, \omega ) \in [0, \infty]$ but  $\int_0^T V_n(p , \omega) dp < \infty$.
Using the previous inequality and \eqref{torta} we obtain 
for  $h, h' \in [s_0, t],$ $ h' < h$
\begin{gather} \label{dii}
\int_{0}^t |f(r+h ) - f(r+h')| dr  \\  \nonumber \le 
C_0  | h'-h|^{(1+ \beta) ({\frac{n - 2d}{n}})} 
\int_{ 0}^{T} 
  V_n(p, \omega) dp 
 + |h-h'|\, \lambda_f(|h-h'|),  
\end{gather} 
where $C_0 = C_0 (\beta, \|b\|_{\beta, T}, \omega, T,x, n,d) >0$.
 Now we choose $n$  large enough such that $(1+ \beta) ({\frac{n - 2d}{n}}) >1$. Dividing by $|h-h'| $ and passing to the limit as $h' \to h^-$ in \eqref{dii}  we find 
$$ 
 \lim_{h' \to h^-} \frac{1}{{|h-h'|} } { \| F(h) - F(h') \|_{L^1 ([0,t]; {\mathbb R}^d)} } =0. 
$$
This shows that there exists the left derivative of $F$ in each $h \in (s_0,t]$ and this derivative is identically  zero on $(s_0,t]$. By the lemma mentioned at the beginning of III Step  we obtain that $F$  is constant. Thus $f$ is constant on $[s_0,t]$ and this finishes the proof.
\end{proof} 
\begin{remark} \label{uniq}
{\em  Note that if $g: [s_0,\tau] \to \R^d$, $\tau = \tau(\omega)\in (s_0, T]$, solves \eqref{integral} on $[s_0, \tau]$ then  we  have  $g(\tau) = \phi(s_0,\tau,x, \omega)$, $\omega \in \Omega'$. Indeed applying (v) on $[s_0, \tau)$ we can use that
$\int_{s_0}^{\tau}b\left(r, g(r)\right)  dr $ $= 
 \int_{s_0}^{\tau}b\left(r, \phi(s_0,r,x, \omega) \right)  dr. $
}
\end{remark}

\begin{remark} \label{uniq1}
{\em  It is a natural question if one can 
 improve \eqref{toro} in Theorem  \ref{d32}.
 A possible stronger assertion could be  the following one:
for each $\alpha \in (0,1)$ and  $N \in \R$ one can find $C(\alpha, T, N, \omega) < \infty$ such that, for any $x,y \in \R^d$, $|x|,$ $|y| <N$,
\begin{gather} 
\label{toro2} \sup_{s \in [0,T]}
\sup_{t \in [s,T]} |\phi (s,t,x, \omega)  - \phi (s,t,y, \omega)| 
\le C(\alpha, T, N, \omega) |x-y|^{\alpha}, \;\; \omega \in  \Omega'.
\end{gather}
This condition is stated as property  4 in Proposition 2.3 of \cite{Sh} for SDEs \eqref{SDE} when $L$ is a Wiener process and  $b \in L^q ([0,T]; L^p(\R^d))$, $d/p + 2/q <1$.

Assuming  $b \in L^{\infty}(0,T; C_b^{0, \beta}(\R^d; \R^d))$
we do not expect that \eqref{toro2} holds in general    when  $L$ and $b$  satisfy Hypotheses  \ref{primo} 
  and  \ref{zero}. Indeed a basic strategy to get \eqref{toro2} when $L$ is a Wiener process is to use the  Kolmogorov-Chentsov  test to obtain a H\"older continuous dependence on $(s,t,x)$;  one cannot use this approach when $L$ is a discontinuous process. 
  Finally  note that the proof of \eqref{toro2} given in \cite{Sh} is not complete
(\eqref{toro2} does not follow directly
 from estimate (4)   in page 5 of \cite{Sh} applying the   Kolmogorov-Chentsov  test). 
}
\end{remark}

Now we present two corollaries of  Theorem \ref{d32} which deal with  SDEs (\ref{SDE})  with possibly unbounded  $b$.

When $b: [0,T] \times {\mathbb R}^d \to {\mathbb R}^d$ is  measurable  and satisfies, for any 
 $\eta \in C_0^{\infty} (\R^d)$,  $b \cdot \eta \in L^{\infty}(0, T; C_{b}^{0, \beta}(\R^d; \R^d))$  we say that  $ b \in L^{\infty}(0, T; C_{loc}^{0, \beta}(\R^d; \R^d))$.
By a   localization procedure we get
\begin{corollary} \label{unb} 
Let  $b \in L^{\infty}(0,T; C_{loc}^{0, \beta}(\R^d; \R^d))$, $\beta \in (0,1],$ and suppose that, for any $\eta \in C_0^{\infty} (\R^d)$, the L\'evy process $L$  and $b \cdot \eta$ satisfy Hypotheses \ref{primo} and \ref{zero}. 

Then there exists an almost sure event $\Omega''$ such that, for any $\omega'' \in \Omega''$, $x \in \R^d$, 
$s_0 \in [0,T)$ and $\tau = \tau(\omega'') \in (s_0, T]$, if $g_1 , g_2 :
 [s_0,\tau ) \to  {\mathbb R}^d$ are c\`adl\`ag solutions of 
\eqref{integral} when $\omega = \omega''$, starting from $x$, then
$g_1(r) = g_2(r)$, $r \in [s_0, \tau)$.
\end{corollary}
\begin{proof} Let $\varphi \in C_0^{\infty}(\R^d)$ be such that $\varphi =1$ on $\{|x|\le 1 \}$ and $\varphi(x) =0$ if $|x|>2$.
 Set $b_n(t,x) = b(t,x) \varphi(\frac{x}{n})$, $t \in [0,T]$, $x\in \R^d$ and $n \ge 1$. Consider for each $n $ an almost sure event $\Omega'_n$ related to $b_n \in L^{\infty}(0,T; C_b^{0, \beta}(\R^d; \R^d))$  
by Theorem \ref{d32};   set 
$\Omega'' = \cap_{n \ge 1} \Omega_n'$. 
 Suppose that 
 $g_1, g_2$ are solutions of \eqref{integral} for a fixed $\omega'' \in \Omega''$.
 Let
 $\tau_k^{(n)} = \tau_k^{(n)}(\omega'')= \inf \{ t \in [s_0,\tau) \, :\, |g_k (t)| \ge n \}$, $k =1,2$ (if $|g_k(s)|< n$, for any $s \in [s_0,\tau )$ then we
 set $\tau_k^{(n)}=\tau$). Define $\tau^{(n)} = \tau_1^{(n)} \wedge \tau_2^{(n)} $ and note that on $\Omega''$ $\tau^{(n)} \uparrow \tau$ as $n \to \infty$. Since on 
 $[s_0, \tau^{(n)}(\omega''))$ both $g_1$ and $g_2$ solve an equation like \eqref{integral} with $b$ replaced by $b_n$ and $\omega = \omega''$
 we can apply  (v)  of Theorem \ref{d32}
 and conclude that $g_1 = g_2$ on $[s_0, \tau^{(n)}(\omega''))$. Since this holds for any $n \ge 1$ we  get that 
$g_1 = g_2$ on $[s_0,\tau (\omega'') )$. 
 \end{proof}

Next we   construct  $\omega$ by $\omega$   strong solutions to \eqref{SDE} when $b$ is possibly unbounded. 
 To simplify we deal with the initial time  $s=0$.
\begin{corollary} \label{cons} 
 Suppose that $L$ and $b$ verify the assumptions of Corollary \ref{unb}. Moreover assume that 
\begin{equation} \label{linear1}
|b(t,x) | \le C (1+ |x|),\;\; x \in \R^d, \; t \in [0,T],
\end{equation} 
for some constant $C>0$. Let  $x \in \R^d$ and $s=0$. Then there exists a (unique) strong solution to \eqref{SDE} starting from $x$.
\end{corollary}
\begin{proof}
 We know that  $t \mapsto L_t(\omega)$ is c\`adl\`ag for any $\omega \in \Omega'$, where   $\Omega'$ is an almost sure event. 
When  $\omega \in \Omega'$  a standard argument based on 
the Ascoli-Arzela theorem
shows that there exists a continuous solution $v=$ $v(\cdot, \omega)$ to $
 v(t) = x + \int_0^t b(s,v(s) + L_s(\omega)) ds
 $ on $[0,T]$.
We define  $v(t, \omega) =0$, if $\omega \not \in \Omega'$, $t \in [0,T]$.
By using the function $\varphi$ as in the proof of Corollary \ref{unb} we  introduce  $b_n(t,x) = b(t,x) \varphi(\frac{x}{n})$, $t \in [0,T]$, $x\in \R^d$ and $n \ge 1$. According to Theorem \ref{d32}  for each $n $ there exists  a  function $\phi_n$ as in \eqref{nuova} and
an almost sure event $\Omega'_n$ corresponding to $b_n$ such that assertions (i)-(v) hold.   
Set 
$\Omega''$ $ = (\cap_{n \ge 1} \Omega_n') \cap \Omega'$. 

Define  $g(t, \omega) = v(t, \omega) + L_t(\omega)$, $t \in [0,T]$, $\omega \in \Omega,$ and 
set 
 $\tau^{(n)} = \tau^{(n)}(\omega)= $ $\inf \{ t \in [0,T) \, :\, |g (t, \omega)| \ge n \}$  (if $|g(s, \omega)|< n$, for any $s \in [0,T )$ then we
 set $\tau^{(n)}(\omega)= T$).  Note that on $\Omega''$ we have $\tau^{(n)} \uparrow T$ as $n \to \infty$.

 Let $\omega \in \Omega''$ and  $n \ge 1$. Since on  $[0, \tau^{(n)}(\omega))$  $g(\cdot, \omega)$   solves an equation like \eqref{integral} with $s_0=0$ and $b$ replaced by $b_{n+k}$, $k \ge 0,$  
we  can apply  (v)  of Theorem \ref{d32} and get that $g(t, \omega) 
= \phi_{n + k} (0,t,x, \omega)$, for any $t \in [0, \tau^{(n)}(\omega))$,  $k \ge 0.$ Since $\tau^{(n)} \uparrow T$ we deduce that, uniformly on compact sets of $[0,T)$, for any $\omega \in \Omega''$, we have  
$
\lim_{n \to \infty} \phi_{n } (0,t,x, \omega)$ $ = g(t, \omega).
$
It follows that $g(t, \cdot)$ is ${\cal F}_t^L$-measurable,  for any $t \in [0,T)$. By setting $g(T, \omega)$ $ =   x + \int_{0}^{T}b\left(r, g(r, \omega)\right)  dr $ 
 $+  L_{T} (\omega)$, we get that 
 $(g(t, \cdot))$ is a strong solution on $[0,T].$
\end{proof}
\begin{remark}{\em The previous condition \eqref{linear1} can be relaxed, by requiring that, for fixed $x \in \R^d$, $s=0$ and $\omega \in \Omega'$, there exists a continuous solution to  the integral equation 
$v(t)$ $ = x + \int_0^t b(s,v(s) + L_s(\omega)) ds$
  on $[0,T]$. The assertion about existence and uniqueness of a strong solution starting from $x$ remains true.}
\end{remark}

\section{Uniqueness for SDEs driven by 
stable  L\'evy processes}

In this section using also results from  \cite{Pr10} and  \cite{Pr14} we show that  Theorem \ref{d32} can be applied
 to a class of SDEs driven by non-degenerate $\alpha$-stable type   processes $L$. Let $s \ge 0$, we are considering
 \begin{equation}
 \label{66}
 X_{t}(\omega) = x + \int_{s}^{t}b\left( X_{u}(\omega) \right)  du \, + \, L_{t}(\omega) - L_s(\omega),
 \quad
\end{equation}
  $x \in {\mathbb R}^d, d \ge 1$, $ t \ge s,$ where $b \in 
  C_b^{0, \beta}({\mathbb R}^d ,{\mathbb R}^d )$, $\beta \in [0,1].$ 
 We  deal with   {\sl pure-jump L\'evy process} $L$ (without drift term), i.e., we assume that the  generating triplet is $(\nu, 0,0)$ (i.e., $Q=0$ and $a=0$ as in  \eqref{d67}).
To state our assumptions on $L$  we   use  the convolution
  semigroup $(P_t)$  associated to
$L$ (or to its L\'evy measure $\nu$) and acting on $C_b({\mathbb R}^d)$, i.e.,  $P_t: 
C_b({\mathbb R}^d) $ $\to C_b({\mathbb R}^d)$, $t \ge 0$, 
 \begin{gather*} 
P_t f(x) = E [f(x+ L_t)] = \int_{{\mathbb R}^d} f(x + z)\,  \mu_t (dz), \;\; t>0, \; f \in
C_b({\mathbb R}^d), \; x \in {\mathbb R}^d,
\end{gather*}
 where $\mu_t$ is the law of $L_t$, and
  $P_0 = I$ (cf. \cite{sato} or \cite{A}).
 The 
 generator $\mathcal L$ of $(P_t)$ is 
 \begin{equation} \label{lll}
 {\cal L}g(x) =
\int_{{\mathbb R}^d} \big(  g (x+y) -  g (x)
   - 1_{ \{  |y| \le 1 \} } \, \lan y , D g (x) \ran \big)
   \, \nu (dy),\;\; x \in {\mathbb R}^d,
 \end{equation}
with $g \in C_0^{\infty}(\R^d)$ (see  Section 6.7 in \cite{A} and Section 31 in  \cite{sato}).
 We now consider the Blumenthal-Getoor index $\alpha_{0} = \alpha_{0}(\nu)$ (see \cite{BG}):
\begin{equation}
\label{bgg}
\alpha_{0} = \inf \Big\{ \sigma >0 \;\; : \;\; { \int}_{\{ |x| \le 1  \, \}} |y|^{\sigma} \nu (dy) \Big \}< \infty;
\end{equation} 
 we always have $\alpha_{0} \in [0,2]$. In the sequel we  require that $\alpha_0 \in (0,2)$.
 Similarly to 
 \cite{Pr14} we make the following assumption on the L\'evy measure $\nu$.
\begin{hypothesis} \label{nondeg1}
{\em  Let $\alpha_0 \in (0,2).$ The convolution  semigroup $(P_t)$ verifies: $P_t (C_b({\mathbb R}^d)) \, \subset
C^1_b({\mathbb R}^d)$, ${t>0}$,
 and, moreover, there exists 
$c_{\alpha_{0}}= c_{\alpha_0}(\nu) >0$ 
  such that
 \begin{align} \label{grad}
 \sup_{x \in {\mathbb R}^d}| D P_t f(x)| \le {c_{\alpha_{0}}}\,{t^{-\frac{1}{ \alpha_{0} } }} \cdot \sup_{x \in {\mathbb R}^d}| f(x)|,\;\; \;
t \in (0,1], \; f \in C_b ({\mathbb R}^d). \qed
\end{align}
}
\end{hypothesis}
Note that Hypothesis \ref{nondeg1}  
implies both Hypotheses 1 and 2 in \cite{Pr14} (taking $\alpha = \alpha_0$). Indeed since  $\alpha_0 \in (0,2)$ we have  
    $
 \int_{\{ |x| \le 1  \, \}} |y|^{\sigma} \nu (dy) < \infty $, 
 for   $\sigma > \alpha_0$. 
  To check  the validity of the gradient estimate \eqref{grad}
we only  mention a
 criterion which is given in \cite{Pr14}; it is based on  Theorem 1.3 in \cite{SSW}.
\bth \label{pri1} Let $L$  be a pure-jump L\'evy process.  
 A sufficient condition in order that \eqref{grad}    holds 
 when $\alpha_0 $ replaced by  $\gamma \in (0,2)$
  is the following one:
 the L\'evy measure $\nu$ of $L$ verifies:
$
 \nu (B) \ge \nu_1(B), $ $ B \in {\cal B}({\mathbb R}^d),
$
where $\nu_1$ is a L\'evy measure on ${\mathbb R}^d$ such that its corresponding symbol
 $
\psi_1 (h)$ $ = - \int_{{\mathbb R}^d} \big(  e^{i \langle h,y \rangle }  - $ $
 1 - \, { i \langle h,y
\rangle} \, 1_{ \{ |y| \le 1 \} } \, (y) \big ) \nu_1 (dy),
$
satisfies, for some positive constants $c_1$, $c_2$  and  $M$,  
\begin{equation} \label{df44}
c_1 |x|^{\gamma} \le Re\, \psi_1 (x) \le c_2 |x|^{\gamma},\;\; \text{when} \; |x|>M.
\end{equation}
\end{theorem}

\begin{example} \label{233} {\em The next examples of $\alpha$-stable type L\'evy processes  are also considered in    \cite{Pr14}. It is easy to check that in each example 
$\alpha_0  = \alpha \in (0,2).$
Thanks to Theorem \ref{pri1} also  \eqref{grad} holds  in each example. 

\hh  Consider the following  L\'evy measure $\tilde \nu$: 
\begin{gather} \label{t66}
\tilde \nu (B) = \int_0^r  \frac{dt}{t^{1+ \alpha}} \int_{S} 1_{B} (t \xi)  \mu (d \xi), \;\; B \in {\cal B}({\mathbb R}^d)
\end{gather}
(cf. Example 1.5 of \cite{SSW} with the index $\beta $ of  \cite{SSW} which is equal to  $ \infty$). Here $r>0$  is fixed;  $\mu$ is a non-degenerate finite non-negative measure on ${\cal B}({\mathbb R}^d)$ with  support on the unit sphere $S$  (non-degeneracy  of $\mu$ is equivalent to say that  its support is not contained in a proper linear subspace of ${\mathbb R}^d$),  $\alpha \in (0,2)$. 
The L\'evy measure   $\tilde \nu$  verifies Hypothesis \ref{nondeg1} since its symbol $\tilde \psi$ verifies \eqref{df44} with $\gamma = \alpha$. This  was already remarked in page 1146 of \cite{SSW}. We only note that, if $h \not =0$, we have 
$$
 Re\, \tilde  \psi (h) =  \int_0^r  \frac{dt}{t^{1+ \alpha}} \int_{S} \Big[ 1 - \cos \Big ( \langle \frac{h}{|h|}, t |h| \,\xi \rangle \Big) \Big] \mu (d \xi).
$$
By changing variable $s =  t |h|$ after some computations one arrives at \eqref{df44}.

Moreover 
Hypothesis 
\ref{zero} holds.
Note that 
$\int_{\{ |x| > 1  \}} \! |y|^{\theta} \, \tilde \nu (dy)\! < \infty$,
 {$\theta \in (0, \alpha)$}.
 Using also  $\tilde \nu$ we find that { the next examples of L\'evy processes verify Hypotheses \ref{zero} and \ref{nondeg1}.}

\hh (i) \textit { $L $ is a non-degenerate 
symmetric $\alpha$-stable process} (see, for instance, \cite{sato} and the references therein).
 In this case $ \nu (B) = \int_0^{\infty}  \frac{dt}{t^{1+ \alpha}} \int_{S} 1_{B} (t \xi)  \mu (d \xi)$, $\, B \in {\cal B}({\mathbb R}^d)$, $\alpha \, \in (0,2)$, where $\mu$ is as in \eqref{t66}.
A standard rotationally invariant $\alpha$-stable process $L$  belongs to this class since its  L\'evy measure has  density  $\frac{c}{|x|^{d+ \alpha}}$ (with respect to the Lebesgue measure in ${\mathbb R}^d$).

 \hh (ii) \textit { $L $ is a   $\alpha$-stable temperated process 
 of special form.}
Here   
$$ \nu (B) =  \int_0^{\infty}  \frac{e^{-t} dt}{t^{1+ \alpha}} \int_{S} 1_{B} (t \xi)  \mu (d \xi), \;\, \; B \in {\cal B}({\mathbb R}^d),
$$
where $\mu$ is as in \eqref{t66}, $\alpha \in (0,2)$. 

Note that in (i) and (ii) we have $\nu (B) \ge e^{-1} \,\tilde \nu (B)$, $B \in {\cal B}({\mathbb R}^d)$, where $\tilde \nu$ is given in \eqref{t66} with $r=1$.

\hh  (iii)  \textit { $L $ is a truncated $\alpha$-stable process.}
 In this case $ \nu (B) = c \int_{\{ |x| \le 1\}}  \frac{1_{B}(x)}{|x|^{d+ \alpha}} \, dx $ $ B \in {\cal B}({\mathbb R}^d), \; \alpha \in (0,2).$

\hh  (iv) \textit{ $L $ is a relativistic $\alpha$-stable process} (cf. \cite{Ry} and see the references therein).
Here $\psi (h) = \big( |h|^2 + m^{\frac{2}{\alpha}} \big)^{\frac{\alpha}{2}} - m$,  for some $m>0$, $\alpha \in (0,2)$, $h \in {\mathbb R}^d$, and so \eqref{grad} holds.  Moreover by Lemma 2 in \cite{Ry} we know that 
 $\nu$ has the  density
$
 { C_{\alpha,d}}{|x|^{-d- \alpha}} \, e^{- m ^{1/\alpha} \, |x|} \, $ $\cdot \, \phi ( m ^{1/\alpha} \, |x|), $ $ x \not = 0,
 $
with $0 \le \phi (s) \le c_{\alpha,d,m} (s^{\frac{d-1+\alpha}{2} } +1)$, $s \ge 0$. Hence $\alpha = \alpha_0$ and 
 also Hypothesis 
\ref{zero} holds 
for any 
 $\theta >0.$
}
\end{example}

\subsection{Results on
strong existence and uniqueness 
by using  solutions of related Kolmogorov equations }

We first present 
 results on strong existence and uniqueness for \eqref{66} 
when $s=0$
which are  special cases of Lemma 5.2 and Theorem 5.3 in \cite{Pr14}. Then we study $L^p$-dependence from the initial condition $x$ following Theorem 4.3 in \cite{Pr10}.
Finally in Theorem \ref{uno11} we will consider the general case when $s \in [0,T]$.

All these theorems do not require the gradient estimates \eqref{grad}. However  they assume  the Blumenthal-Getoor index $\alpha_0 \in (0,2)$,
 $b \in C_b ({\mathbb R}^d, {\mathbb R}^d) $  and  classical  solvability of the following  
 Kolmogorov type equation:
\begin {equation} \label{eq1}
 {\lambda}  u (x)-
 {\mathcal L} u(x) \,  - \, Du(x) \, b(x) \, = \, b(x), \;\; {x \in {\mathbb R}^d},
\end{equation}
 where $b : {\mathbb R}^d \to {\mathbb R}^d$ is given in  \eqref{66},
  ${\cal L}$ in \eqref{lll}
 and
 $\lambda>0$; {the equation is intended componentwise,} i.e., $u : {\mathbb R}^d \to {\mathbb R}^d$ and, setting ${\cal L}_b = {\cal L} + b(x) \cdot D$,
   \begin{gather} \label{dee1}
 \lambda u_k(x) -
 {\cal L}_b u_k (x) =
  b_k(x), \;\;  k =1, \ldots, d, 
 \end{gather}
with $u(x)$ $ = (u_k(x))_{k=1, \ldots, d}$ and  $b(x) = (b_k(x))_{k=1, \ldots, d}$. The approach to get strong uniqueness 
passing through  solutions to \eqref{eq1} is similar to the one used in Section 2 of \cite{FGP} (see also \cite{Ver}).

 Remark
  that  ${\mathcal L}
 g (x)$
 in \eqref{lll} is  well defined even  for 
 $g \in C_b^{1+ \gamma} ({\mathbb R}^d)$ if
 $\alpha_0 < 1+ \gamma $ and  $\gamma \in [0,1)$ (cf. formula (13) in \cite{Pr14}).  
 Indeed when $|y| \le 1$  we can use the bound
 $
  | g(y + x) - g(x)
   -  \, y \cdot D g (x) |$ $
 \le  [ Dg]_{\gamma} \, |y|^{1+  \gamma},
  $ $ x \in {\mathbb R}^d.
$

\smallskip
  In addition
  $ {\mathcal L} g
 \in \,  C_b({\mathbb R}^d)$ when  $g \in  C^{1+ \gamma}_b ({\mathbb R}^d)$ and
 $1+ \gamma > \alpha_0$.
The next result is 
stated in 
Theorem 5.3 of \cite{Pr14} in a more general form which also shows  the   differentiability
 of solutions with  respect to  $x$
and the homeomorphism  property.
\begin{theorem}
\label{uno1}
 Let $L$ be any L\'evy process on $(\Omega,{\cal F}, P)$ with generating triplet  $(\nu, 0,0)$  such that $\alpha_0 = \alpha_0(\nu) \in (0,2)$
(see \eqref{bgg})
 and let $b \in C_b ({\mathbb R}^d, {\mathbb R}^d)$ in  \eqref{66}.
  Suppose 
 that, for some $\lambda >0$,  there exists
 $u= u_{\lambda} \in C^{1+ \gamma}_b ({\mathbb R}^d, {\mathbb R}^d)$, $\gamma \in (0,1)$ and $2 \gamma > {\alpha_0}$,
 which solves  
 \eqref{eq1}.
  Moreover,
  assume  
   $ {\| Du_{\lambda} \|_{0}  < 1/3} $. 
    
Then
 on $(\Omega, {\cal F}, P)$, for any $x \in {\mathbb R}^d$, there exists a pathwise unique  strong solution $(X_t^x)_{t \ge 0}$ to \eqref{66} when $s=0$.
\end{theorem}
Next we formulate a special case of  Lemma 5.2 in \cite{Pr14}. 
It  uses  the stochastic integral
  against the compensated Poisson random measure  $\tilde N$
 (see, for instance, \cite{Ku}).
\begin{lemma} \label{dee} 
Under the same hypotheses of Theorem \ref{uno1} let $T>0$ and suppose that 
   $(X_t^x)_{t \in [0,T]}$  is  a strong  solution of
 \eqref{66} on $[0,T]$ when $s=0$  (starting from $x \in {\mathbb R}^d$),
 then, using $u_{\lambda}$ of Theorem \ref{uno1}, we have, $P$-a.s.,
 for any $t \in [0,T]$, 
\begin{gather} 
\label{iii}
 u_{\lambda} (X_t^x ) -  u_{\lambda} (x )
\\ \nonumber
=  x + L_t - X_t^x   + \lambda  \int_0^t  u_{\lambda}(X_s^x) ds
 +  \int_0^{t} \int_{{\mathbb R}^d \setminus \{ 0 \} } \! \! [  u_{\lambda}(X_{s-}^x + y) -  u_{\lambda}(X_{s-}^x)]
   \tilde N(ds, dy).
\end{gather} 
\end{lemma}
\begin{proof} 
The assertion is stated in   Lemma 5.2 of \cite{Pr14} 
  for  weak solutions $(X_t^x)_{t \ge 0}$ with the  condition $1+ \gamma > {\alpha_0}$, $\gamma \in (0,1]$.
 Clearly 
  such lemma  works also for strong solutions $(X_t^x)_{t \in [0,T]}$ which solves \eqref{66} on $[0,T]$ (the proof is based on It\^o's formula for $ u_{\lambda}(X_t^x))$;
  further the condition $2 \gamma > {\alpha_0} $ of Theorem \ref{uno1} implies $1+ \gamma > {\alpha_0}$. 
\end{proof}
  To prove Davie's uniqueness for \eqref{66} 
we   need the following $L^p$-continuity of the solutions  w.r.t. initial conditions.
\begin{theorem} 
\label{sti1} Under the same hypotheses of Theorem \ref{uno1}    
 let $T>0$, $s=0$, and  
  consider  two strong solutions $(X_t^x)_{t \in [0,T]}$ and 
$(X_t^y)_{t \in [0,T]}$
  of \eqref{66} on $[0,T]$ which are defined on  $(\Omega, {\cal F}, P)$,  starting from $x$ and $y \in {\mathbb R}^d$ respectively. 
    For any  $t \in [0,T]$, $p \ge 2$,   we have 
 \begin{equation} \label{ciao223}
\begin{array} {l}
E\big[ \sup_{0 \le s \le t} |X_s^x - X^y_s|^p \big] \le C(t) \,
 |x-y|^p,
\end{array}
  \end{equation}
 with  $C(t) = C (t,  \nu, p,  \lambda, d,  \gamma, 
 \| u_{\lambda}\|_{C^{1+ \gamma}_b}) >0$  which is independent of $x$ and $y$; here $u_{\lambda}$ is as in Theorem \ref{uno1} (further $C (t,  \nu, p,  \lambda, d,  \gamma, \cdot )$ is increasing).
\end{theorem}  
\begin{proof} The proof follows  the one of  (i) in  Theorem 4.3 of \cite{Pr10}. We  only give a sketch of the proof here.
 We set $X= X^x$, $Y = X^y$ and $u = u_{\lambda}$. 
 We have from Lemma \ref{dee}, $P$-a.s., using that 
$\| Du\|_0 \le 1/3 $,  $| X_t \,  - \, Y_t| \le \frac{3}{2} \big ( \Gamma_{1}(t) \, +
 \Gamma_{2}(t) \, + \Gamma_{3}(t)
 \, + \Gamma_{4} \big),$ {where}
 \begin{gather*}
  \Gamma_{1}(t) =  \Big|
 \int_0^t \! \!\int_{ \{ |z| > 1\} } [ u(X_{s-} + z) - u(X_{s-})
  - u(Y_{s-} + z) + u(Y_{s-})]
   \tilde N(ds, dz) \Big|,
\\
\Gamma_{2}(t) =  \lambda  \int_0^t   |u(X_s)- u(Y_s)| ds,
\\
\Gamma_{3}(t) =  \Big| \int_0^t \int_{\{ |z| \le  1\} } [ u(X_{s-} +
z) - u(X_{s-})
 - u(Y_{s-} + z) + u(Y_{s-}) ]
   \tilde N(ds, dz) \Big|,
\end{gather*}
$ \Gamma_{4}  \, = \, |u(x) - u(y)| + |x-y| \le \, \frac{4}{3}|\, x-y| $. Remark 
that, $P$-a.s.,
$$
\begin{array}{l}
\sup_{0 \le r \le t } |X_r \, - \, Y_r|^p \le {C_1} \, |x-y|^p +
 C_1\sum_{j=1}^3 \sup_{0 \le r \le t } \, \Gamma_j(r)^p.
\end{array}
$$
By the H\"older inequality,
 $
\sup_{0 \le r \le t } \Gamma_{2}(r)^p \le C_2  \, t^{p-1} \int_0^t
\sup_{0 \le s \le r } |X_s \, - \, Y_s|^p \, dr,
$
 where $C_2 = C_2 (p, \lambda, \| u_{\lambda}\|_{C^{1+ \gamma}_b})$. To estimate $\Gamma_{1}$ and $\Gamma_{3}$ we use 
 $L^p$-estimates for stochastic integrals against $\tilde N$
 (cf. \cite[Theorem  2.11]{Ku}
 or the proof of  Proposition 6.6.2 in \cite{A}). 

We find, since
 $ | u(X_{s-} + z)
  - u(Y_{s-} + z) + u(Y_{s-} ) - u(X_{s-}) |
 $ $\le \frac{2}{3} |X_{s-}
  -  Y_{s-}|,
  $  setting  $A = \{ |z| >1\}$,
\begin{gather*}
E[\, \sup_{0 \le r \le t } \Gamma_{1}(r)^p] 
\\ \le C_3
 E \Big [ \! \Big ( \int_{0}^{t} ds \int_{ A} | \!  u(X_{s-} + z)
  - u(Y_{s-} + z) + u(Y_{s-})- u(X_{s-}) |^2 \nu(dy) \Big )^{p/2}
  \Big ]
\\ + \, C_3
 E \, \int_{0}^{t} ds \int_{ A } | u(X_{s-} + z)
  - u(Y_{s-} + z) + u(Y_{s-})- u(X_{s-})|^p \nu(dy)
\\  \le \, C_4
\, (1+ t^{p/2
-1}) \, 
 \int_0^t E[
\sup_{0 \le r \le s } |X_r - Y_r|^p ] ds,
\end{gather*}
where $C_3 =
 \int_{\{ |z|>1 \} } \, \nu(dz) \,  +  \big(\int_{\{ |z|>1\} }
\nu(dz) \, \big)^{p/2} $.
 To treat  $\Gamma_{3}$ we need 
the hypothesis $2 \gamma >
{\alpha_0}$. By $L^p$-estimates of stochastic integrals and using Lemma 4.1 in \cite{Pr10}   we get
 \begin{gather*}
E[\sup_{0 \le r \le t } \Gamma_{3}(r)^p]
 \le C_5 \| u\|_{C^{1+ \gamma}_b }^p
 \, E \Big [
 \Big ( \int_0^t dr \int_{\{  |z| \le  1 \} } |X_{r} - Y_r|^2
 |z|^{2\gamma}
 \nu(dz) \Big )^{p/2} \Big ]
\\
 + \, C_5 \| u\|_{C^{1+ \gamma}_b }^p \; 
 \, E  \int_0^t |X_{r} - Y_r|^p dr \int_{\{  |z| \le  1 \} } 
 |z|^{\gamma p}
 \nu(dz).
\end{gather*}
Note that $\int_{\{ |z| \le  1 \} }  
|z|^{p \gamma}
   \nu(dz) < \infty  $,
 since $p \ge 2$ and $2 \gamma > {\alpha_0}$.
  Collecting the previous estimates, we arrive at
\begin{gather*}
E[ \sup_{0 \le r \le t} |X_r - Y_r|^p] \le C_6 \,
 |x-y|^p \,+\, C_6  \, (1+ t^{p -1})
 \, \int_0^t E[\sup_{0 \le
r \le s }  |X_r - Y_r|^p] \, ds,
\end{gather*} 
 $C_6 = C_6 (\nu $, $p$, $\lambda,$ $d, \gamma)>0$.
By the
  Gronwall lemma we obtain the assertion 
  with $C(t) = C_6 \exp \big( C_6  \, (1+ t^{p -1})\big)$.
\end{proof}
As a  consequence of the previous results we get 
\begin{theorem}
\label{uno11} Under the same hypotheses of Theorem \ref{uno1}    
 let $T>0$ and $s\in [0,T]$. Then, for any $x \in {\mathbb R}^d$, 
there exists a pathwise unique strong solution   ${\tilde X}^{s,x} 
=({\tilde X}^{s,x}_t)_{t \in [0,T]}$ to \eqref{66} on   
 $(\Omega, {\mathcal F}, P)$ (recall that ${\tilde X}_t^{s,x} =x$ for $t \le s$).
 Moreover  if ${ U}^{s,x}$ and ${ U}^{s,y}$ are two strong solutions on $[0,T]$ defined on $(\Omega, {\cal F}, P)$  and starting at $x$ and $y$, then  we have, for  $p \ge 2$,
 \begin{equation} \label{ciao22}
\sup_{s \in [0,T]}  \, E[\sup_{ s \le t \le T} |{U}_t^{s,x} - {U}_t^{s,y}|^p] 
\le C(T) \,
 |x-y|^p, \;\;\;  x,\, y \in {\mathbb R}^d,
  \end{equation}
where $C(T) = C (T,  \nu, p,  \lambda, d,  \gamma, 
 \| u_{\lambda}\|_{C^{1+ \gamma}_b}) >0$  as in \eqref{ciao223}.
 \end{theorem}
\begin{proof} 
 \underline{\it Existence.}  Let us fix $s \in [0,T]$ and consider  the
 new process  $L^{(s)} = (L^{(s)}_t)$ on $(\Omega, {\cal F}, P)$,  $L^{(s)}_t =L_{s+ t} - L_s$, $t \ge 0$. This  is a  L\'evy process 
  with the same generating triplet  of $L$
and is independent of ${\cal F}_s^L$ (see Proposition 10.7 in \cite{sato}).
According to Theorem \ref{uno1} there exists a unique strong solution 
to 
\begin{equation} \label{dedo}
X_{t} = x + \int_0^{t}   b(X_r) dr +  L_{t}^{(s)},\;\;\; t \ge 0,
\end{equation}
which we denote by  $(X_{t, L^{(s)}}^{x    })$
to stress its dependence on  $L^{(s)}$.
  Note that, for any $t \ge 0$,  $X_{t, L^{(s)}}^{x    }$ is measurable with respect to ${\cal F}_t^{L^{(s)}} = {\cal F}_{s,t +s}^L$.
 Let us define a new  process with c\`adl\`ag paths 
 $({\tilde X}^{s,x}_t)_{t \in [0,T]}$,
\begin{equation}
\label{strong1}
{\tilde X}^{s,x}_t =  X_{t-s, L^{(s)}}^{x    },\;\; \text{for} \;   s \le t \le T; \;\;\; {\tilde X}^{s,x}_t =x,\;\; 0\le t \le s.
\end{equation} 
Writing $V_t = {\tilde X}^{s,x}_t$, $t \in [0,T]$, to simplify notation, we note that
 $V_t$ is ${\cal F}_{s,t}^L$-measurable, $t \ge s$. Moreover it solves  equation \eqref{66}; indeed, for  $t \in [s,T]$, 
$$
V_t = X_{t-s, L^{(s)}}^{x    } = x + \int_0^{t-s}  b( X_{r, L^{(s)}}^{x    }) dr +  L_{t} - L_s
= x + \int_s^{t}  b(V_{r}) dr +  L_{t} - L_s.
$$
\underline{\it Uniqueness.} Let $(U^{s,x}_t)$ 
be another strong solution.
 We have, $P$-a.s., for $s \le t \le T$,
\begin{gather*}
U_{t-s +s}^{s,x} = 
   x + \int_s^{t}  b(U_{r}^{s,x}) dr +  L_{t} - L_s
\\=
x + \int_0^{t-s}  b(U_{r+s}^{s,x}) dr +  L_{t} - L_s
= x + \int_0^{t-s}   b(U_{r+s}^{s,x}) dr +   L_{t-s}^{(s)}.
\end{gather*}
Hence $(U_{r +s}^{s,x})_{r \in [0, T-s]}$ solves \eqref{dedo} 
on $[0, T-s]$. By \eqref{ciao223} we get 
$$P (U_{r +s}^{s,x} = X^x_{r, L^{(s)}}, \, r \in [0, T-s])= P (U_{r +s}^{s,x} = \tilde X^{s,x}_{r +s}, \, r \in [0, T-s])=1.$$ 
This shows the assertion.

\vskip 1mm \hh \underline{\it $L^p$-estimates.} We have for any fixed $s \in [0,T]$,
   $p \ge 2$, 
$E[ \sup_{ s \le t \le T} |{U}_t^{s,x} - {U}_t^{s,y}|^p] 
= 
E[\sup_{ s \le t \le T}  |X_{t-s , \, L^{(s)}}^{x } - X_{t-s , \, L^{(s)} }^y|^p]$ by uniqueness. 
Using   \eqref{ciao223} we get
\begin{gather*}
\sup_{s \in [0,T]} E[ \sup_{ s \le t \le T} |{U}_t^{s,x} - {U}_t^{s,y}|^p] 
= 
\sup_{s \in [0,T]} E[\sup_{ s \le t \le T}  |X_{t-s , \, L^{(s)}}^{x } - X_{t-s , \, L^{(s)} }^y|^p] 
\\
\le \sup_{s \in [0,T]}E[\sup_{t \in [0,T]}  |X_{t , \, L^{(s)}}^{x } - X_{t, \, L^{(s)} }^y|^p]  
\le C(T) \,
 |x-y|^p. 
 \end{gather*}
\end{proof}

\subsection { A  Davie's type uniqueness result  when 
 $\alpha_0 \in [1, 2)$ 
 }
 
Here we prove a   Davie's type uniqueness result for \eqref{66} (cf. Theorem \ref{d32}).
 We consider  the Blumenthal-Getoor index ${\alpha_0} \in [1,2)$ (see \eqref{bgg}) and assume as in \cite{Pr10} and \cite{Pr14} that  $b \in C_b^{0, \beta}({\mathbb R}^d, {\mathbb R}^d)$  with 
 $\beta \in  \big( 1 - \frac{{\alpha_0}}{2} , 1 \big]$.

 To check Hypothesis \ref{primo} we will use Theorem
\ref{uno11}
 and the following purely analytic result
 (see Theorem 4.3 in \cite{Pr14}; its the proof follows the one in Theorem 3.4 of \cite{Pr10}). Note that the next hypothesis  $\alpha_0 + \beta <2$ could be dropped. Moreover,  
to simplify we have only considered the case $\lambda \ge 1$ instead of $\lambda >0$.  
\begin{theorem} 
\label{reg} Assume Hypothesis \ref{nondeg1}
  with ${\alpha_0} = \alpha_0(\nu) \ge 1$.  Let $0< \beta  < 1  $ with $ {\alpha_0} + \beta \in (1,2)$ and consider $\cal L$ in \eqref{lll}.
 Then,  for any
 $\lambda \ge 1$, $f \in C^{\beta}_b ({\mathbb R}^d)$,
   there exists a unique  solution
  $ w_{\lambda} \in C^{{\alpha_0} + \beta}_b
  ({\mathbb R}^d)$ to 
  \begin{gather} \label{mia}
 \lambda w(x) -  {\mathcal L} w(x)  - b(x) \cdot Dw(x) = f(x), \;\; x \in  {\mathbb R}^d
 \end{gather}
  Moreover,
  there exists  $C_0 $ $= C_0(\alpha_0 (\nu), $ $ d, \beta, \| b\|_{C^{\beta}_b}, \nu) >0$ such that
 \begin{equation} \label{sch4}
 \begin{array}{l}
 \lambda \| w_{\lambda}\|_0 + [
Dw_{\lambda}]_{C^{{\alpha_0} + \beta - 1}_b}
  \le C_0 \| f\|_{C^{\beta}_b}, \;\;  \lambda \ge 1.
\end{array} 
\end{equation}
 Finally,  we have 
 $ \| Dw_{\lambda} \|_{0} < 1/3$, for any  $\lambda  \ge \lambda_0( d, \| b\|_{C^{\beta}_b}, {\alpha_0} (\nu), \beta, \nu)\ge 1$.
\end{theorem}
\begin{proof} We only make some comments on  $C_0$ and $\lambda_0$.
 Let us first consider $C_0$. 
 To see that $C_0 =  C_0({\alpha_0}(\nu), d, \beta, \| b\|_{C^{\beta}_b}, \nu)$ 
 we look into the proof of Theorem 4.3 in \cite{Pr14}. In such proof the Schauder estimates \eqref{sch4} are first established as apriori estimates by a localization procedure. 
 This method is based on Schauder estimates
already proved   in the constant coefficients case, i.e.,  when $b(x) =k$, ${x \in {\mathbb R}^d}$ (see Theorem 4.2 in \cite{Pr14}). 
The Schauder constant $C_0$ depends on the Schauder constant  $c$ appearing in formula (16) of Theorem 4.2 in \cite{Pr14} when  ${\lambda \ge 1}$. Such constant $c$ depends on ${\alpha_0}(\nu), \beta, d$  and also on the constant $c_{{\alpha_0}}$ of  the gradient estimates \eqref{grad} (see, in particular, estimates (18)-(21) in the proof of Theorem 4.2 in \cite{Pr14}). 

\noindent Let us consider $\lambda_0$. Recall  the  simple estimate  
$ \| Dw_{\lambda}\|_{0} $ $\le N
[Dw_{\lambda}]_{C^{{\alpha_0} + \beta - 1}_b}^{\frac{1}{{\alpha_0} + \beta}}
\, \| w_{\lambda} \|_{0}^{ \frac{ {\alpha_0} + \beta -1 }{{\alpha_0} + \beta}} ,$ where $N =N({\alpha_0}, \beta, d)$ (cf. the  proof of Theorem 3.4 in \cite{Pr10}).
By  \eqref{sch4}  we get
$  \| Dw_{\lambda}\|_{0} 
   $ $\le N  C_0  \, \lambda^{- {\frac{{\alpha_0} + \beta \, - 1 }{{\alpha_0} +
\beta}}} \; \| f\|_{C^{\beta}_b}, $ $\lambda \ge 1$, and the  assertion follows by choosing  $\lambda_0 > 1 \vee 
(3N  C_0)^{\frac{{\alpha_0} + \beta  }{{\alpha_0} + \beta -1}}$.
\end{proof}

Currently we do not know if the statements in Theorem \ref{reg}
hold also  when ${\alpha_0}  \in (0,1)$ (maintaining all the other assumptions). 

Now we apply Theorem \ref{d32} to get   Davie's type uniqueness  for the SDE \eqref{66}.
\begin{theorem} \label{main}
Let $L$ be a $d$-dimensional L\'evy process on $(\Omega, {\cal F}, P)$ with generating triple $(\nu, 0,0)$ satisfying Hypothesis \ref{nondeg1} with ${\alpha_0} \in [1,2)$. Suppose also that 
$ \int_{\{ |x| > 1  \, \}} |y|^{\theta} \nu (dy) < \infty,
 $ for some $\theta >0.$
Let us consider \eqref{66} with  
 $b \in C_{b}^{0,\beta}\left(  \mathbb{R} ^{d} ; {\mathbb R}^d \right) $
  and  ${ \beta \in \big( 1 - \frac{{\alpha_0}}{2},1 } \big]$.

Then $L$ satisfies Hypothesis \ref{primo} and, for any $T>0$, there exists a function $\phi$ 
as in Theorem \ref{d32}  such that  assertions (i)-(v)  hold on some almost sure event $\Omega'$.
 \end{theorem}
\begin{proof}
When $\beta =1$ Hypothesis \ref{primo} is clearly satisfied. Let us consider $\beta \in \big( 1 - \frac{{\alpha_0}}{2},1  \big).$  Since $C_{b}^{\beta'} (  \mathbb{R} ^{d} , {\mathbb R}^d )
  \subset C_{b}^{\beta} (  \mathbb{R} ^{d} , {\mathbb R}^d ) $
 when $0 < \beta \le \beta' \le 1$,  we may  assume that  $1 - \frac{\alpha_0}{2} < \beta <  2 - {\alpha_0}$.
 To verify
 Hypothesis \ref{primo}  we  use   Theorems \ref{reg}
 and 
\ref{uno11}. By
   Theorem \ref{reg} we have
   a solution $u_{\lambda}
   \in C^{1+ \gamma}_b ({\mathbb R}^d, {\mathbb R}^d)$
   to  \eqref{eq1} with
    $\gamma = {\alpha_0} - 1 + \beta \in (0,1)$ for any $\lambda\ge 1$.  Note that  $
  2 \gamma$ $ = 2{\alpha_0} - 2 + 2 \beta$ $ > {\alpha_0}.$
 Choosing $\lambda = \lambda_0(d, \| b\|_{C^{\beta}_b}, {\alpha_0}(\nu), \beta) $
    we obtain that also $\| Du_{\lambda}\| < 1/3$ holds. 
  
 Using  Theorem
\ref{uno11} we can check the validity of 
   \eqref{ciao221}. Note that 
 the constant $C(T)$ appearing in \eqref{ciao22}
 depends on $T$,  $\nu$, $p$, ${\alpha_0} (\nu)$, $\lambda,$ $d$,  $\gamma$
 and $\| u_{\lambda}\|_{C^{1+ \gamma}_b}$. However 
 by Theorem \ref{reg} 
 $\gamma = {\alpha_0} - 1 + \beta$, 
  $\lambda = \lambda_0( d, \| b\|_{C^{\beta}_b}, {\alpha_0}, \beta)$ 
 and 
 $\| u_{\lambda}\|_{C^{1+ \gamma}_b}$ $ = \| u_{\lambda}\|_{C^{{\alpha_0} + \beta}_b} $ $\le N({\alpha_0}, \beta, d) \, C_0
  \, \|b \|_{C^{\beta}_b}
 $
where $C_0$ appears in the Schauder estimates \eqref{sch4}. It follows that $C(T)$ in \eqref{ciao22} has the right dependence on $d, p, $ $ \beta, \nu$, $\|b \|_{C^{\beta}_b}$ and  $T$ as required in   \eqref{ciao221}.
To finish the proof we  
  apply Theorem \ref{d32} since Hypotheses \ref{primo} and   \ref{zero} hold.
 \end{proof} 
 
\begin{remark} {\em Theorem \ref{main}  shows that under suitable  assumptions on $L$ and $b$  Davie's uniqueness (or path-by-path uniqueness)  holds for the SDE \eqref{SDE}.
 Moreover, the unique strong solution is given by a function $\phi$ which satisfies all the assertions of  Theorem \ref{d32}, 
 including \eqref{sd} and \eqref{toro},
for any $\omega \in \Omega'$, where $\Omega'$ 
is an almost sure event 
independent of $s,t$ and $x$.   There are no similar  results in the  literature on stochastic flows for SDEs \eqref{SDE}  driven by stable type processes  (cf.
 \cite{Pr10}, \cite{Pr14} and  the recent paper \cite{CSZ} which contains the most general available results about  existence and $C^1$-regularity of   
 stochastic flow). 
}
\end{remark}

\subsection{ Davie's type uniqueness   when 
  $\alpha_0 = \alpha \in (0,1)$}

Here we only consider the SDE \eqref{66} when 
 $L = L_{\alpha}$ is a symmetric rotationally  invariant $\alpha $-stable process with $\alpha \in (0,1)$ (the case of $\alpha \in [1,2)$ is already treated   in Theorem \ref{main}). For each $\alpha \in (0,1)$ its L\'evy measure $\nu = \nu_{\alpha}$ has density 
 $ \frac{ c_{\alpha,d}}{|y|^{d+ \alpha}},$ $  y \not =0,$
 and its generator ${\cal L} = {\cal L}^{(\alpha)}$  (see \eqref{lll}) 
   coincides with the
 fractional Laplacian
 $-(- \triangle)^{\alpha/2}$
 (see Example 32.7 in \cite{sato}). Note that, for any $g \in C^1_b({\mathbb R}^d)$, the mapping: 
\begin{equation} \label{lll1}
x \mapsto  {\cal L}g(x) =
c_{\alpha, d}\int_{{\mathbb R}^d} \frac{   g (x+y) -  g (x) }  {|y|^{d+ \alpha}}
   \,  dy  \;\; \text{belongs to $C_b({\mathbb R}^d)$.}
\end{equation} 
Clearly  $\alpha = \alpha_0$ (see \eqref{grr}).  Using 
Theorem
\ref{uno11}
of the previous section together with Theorem \ref{sil}
 we can apply    Theorem \ref{d32} and obtain
 \begin{theorem} \label{hp}
Let $L$ be a $d$-dimensional symmetric rotationally  invariant $\alpha $-stable process with $\alpha \in (0,1)$ defined on 
$(\Omega, {\cal F}, P)$. 
 Let us consider the SDE \eqref{66} with  
 $b \in C_{b}^{0, \beta}\left(  \mathbb{R} ^{d} ; {\mathbb R}^d \right) $
  and  ${ \beta \in \big( 1 - \frac{\alpha}{2},1 } \big]$.

Then $L$ satisfies Hypotheses \ref{primo} and \ref{zero} and, for any $T>0$, there exists a function $\phi$ as in Theorem \ref{d32}  such that  assertions (i)-(v)  hold on some almost sure event $\Omega'$.
 \end{theorem} 
 We first state
 a result which is related to Theorem \ref{reg}.  It shows sharp $C_b^{\alpha +\beta}$-regularity of  solutions to 
\eqref{mia}.
The proof 
 is based on  
 Theorem 1.1 in \cite{Si}. 
 \begin{theorem} 
\label{sil} Let us consider the fractional Laplacian ${\cal L} $ given in \eqref{lll1} with $\alpha \in (0,1)$. Let $\beta  \in (0, 1)$
 such that $\alpha + \beta >1$.
 Then,  for any
 $\lambda \ge 1$, $f \in C^{\beta}_b ({\mathbb R}^d)$,
   there exists a unique  solution
  $w= w_{\lambda} \in C^{\alpha + \beta}_b
  ({\mathbb R}^d)$ to \eqref{mia}.
 Moreover, 
  there exists  $C_0 $ $= C_0(\alpha, $ $ d, \beta, \| b\|_{C^{\beta}_b}) >0$ such that
 \begin{equation} \label{sch44}
\lambda \| w_{\lambda}\|_0 + [Dw_{\lambda}]_{C^{\alpha + \beta - 1}_b} 
 \le C_0 \| f\|_{C^{\beta}_b}, \;\;  \lambda \ge 1.
 \end{equation}
 Finally,  we have 
 $ \| Dw_{\lambda} \|_{0} < 1/3$, for any  $\lambda  \ge \lambda_0$, with $\lambda_0( d, \| b\|_{C^{\beta}_b}, \alpha , \beta)\ge 1$.
\end{theorem}
\begin{proof}  The uniqueness  follows by the maximum principle (see Proposition 3.2 in \cite{Pr10} or Proposition 4.1 in \cite{Pr14}) which states that $\lambda \| w_{\lambda}\|_0 \le  \| f\|_0$. Let ${\cal L}_b$ be the fractional Laplacian ${\cal L}$ plus the drift $b$ (i.e., 
 ${\cal L}_b$ $= {\cal L} + b \cdot D$). The proof proceeds in some steps.

\hh \textit{I step.} Let $\lambda \ge 1$.
 We provide apriori estimates for classical $C^{1}_b$-solutions $u$ to $\lambda u - {\cal L}_b u =f$ on ${\mathbb R}^d$ (with $f \in C_b^{\beta}({\mathbb R}^d)$, $b \in C_b^{\beta}({\mathbb R}^d; {\mathbb R}^d)$ and $\alpha + \beta >1$).

Let $u = u_{\lambda} \in C^{1}_b({\mathbb R}^d)$ be a solution to $\lambda u -  {\cal L}_b u =f$ on ${\mathbb R}^d$; in the sequel we will 
 consider 
open balls $B_r(x_0)$ of center $x_0 \in {\mathbb R}^d$ and radius $r>0$. 
  Let $x_0 \in {\mathbb R}^d.$
One can define   $v(x) = u(x+ x_0)$, $x \in {\mathbb R}^d$. 
 Since ${\cal L}v(x)$ $= {\cal L}u (x+ x_0)$, $x \in {\mathbb R}^d,$ 
 we get that  $v \in C^1_b({\mathbb R}^d)$ solves $\lambda v -  {\cal L}_{b_0} v = f_0$
on ${\mathbb R}^d$ where ${\cal L}_{b_0}$ has the  drift $b_0(\cdot) = b( \cdot + x_0 )$ and $f_0 (\cdot) = f(\cdot + x_0) $.

Setting $\tilde v(t,x) =e^{\lambda t} v(x)$, $\tilde f_0(t,x) = e^{\lambda t} f_0(x)$, $t \in [-1,0]$, $x \in {\mathbb R}^d$, we see  that $\tilde v$ is  a bounded solution of  
$$
\partial_t \tilde v - {\cal L}_{b_0} \tilde v = \tilde f_0
 \;\; \text{on}\;\;  [-1,0] \times B_1(0) 
$$
according to the  definition  of viscosity solution given at the  beginning of  Section 3.1 in \cite{Si}.
Hence we can apply  Theorem 1.1 in \cite{Si} to $\tilde v$. Recall that in the Silvestre notations his  $s \in (0,1)$ is our  $\alpha/2$ and his $\alpha \in (0, 2s)$ corresponds with  our $\alpha + \beta -1$.   
We deduce by \cite{Si} that $\tilde v(t, \cdot) \in C^{\alpha + \beta}(B_{1/2}(0))$
 and moreover
 \begin {gather*}
 \|v \|_{C^{\alpha + \beta}(B_{1/2}(0))} = \|\tilde v \|_{ L^{\infty}([-1/2,0]; C^{\alpha + \beta}(B_{1/2}(0))}
\\ \le 
C_2 (\| \tilde v \|_{ L^{\infty}([-1,0] \times {\mathbb R}^d)} + 
\|\tilde f  \|_{L^{\infty}([-1,0]; C^{\beta}(B_{1/2}(0))} )
=
C_2 (\|v \|_{0}  +   \|f_0 \|_{C^{\beta}_b({\mathbb R}^d)} ), 
\end{gather*} 
where $C_2$  depends only on $\| b_0 \|_{C^{\beta}_b({\mathbb R}^d; {\mathbb R}^d)}
$ $ = \| b \|_{C^{\beta}_b({\mathbb R}^d; {\mathbb R}^d)}$, $\alpha$ and $d$ and is independent of $\lambda$. 
 Thus we get that $u_{\lambda} \in C^{\alpha + \beta}(B_{1/2}(x_0))$
 with a bound for the $C^{\alpha + \beta}$-norm of $u_{\lambda} $  on 
$B_{1/2}(x_0)$ by the quantity $ C_2 (\|u_{\lambda} \|_{0}  +   \|f \|_{C^{\beta}_b({\mathbb R}^d)} )$.
Since $C_2$ is independent on $x_0$   it is clear that we have  
$u_{\lambda} \in  C^{\alpha + \beta}_b({\mathbb R}^d)$ (cf. for instance page 434 in \cite{Pr10}) and the following estimate holds with ${C_3 = C_3 (\| b \|_{C^{\beta}_b}, \alpha, d, \beta)>0}$
 \begin{equation*}
\begin{array}{l}
 \|u_{\lambda} \|_{C^{\alpha + \beta}_b({\mathbb R}^d)} \le C_3 (\|u_{\lambda} \|_{0}  +   \|f \|_{C^{\beta}_b({\mathbb R}^d)} ).
\end{array}
\end{equation*}
By Proposition 3.2 in \cite{Pr10} we know that $\lambda \| u_{\lambda}\|_0
 \le  \| f\|_0$. Hence we arrive at 
\begin{equation} \label{silvv}
 \|u_{\lambda} \|_{C^{\alpha + \beta}_b({\mathbb R}^d)} \le 2C_3   \|f \|_{C^{\beta}_b({\mathbb R}^d)},\;\;\; \lambda \ge 1.
\end{equation}
\hh \textit{II step.} Let $\lambda \ge 1$. We show the existence of  a $C^1_b$-solution to  $\lambda w - {\cal L}_b w= \tilde f$ when $b \in C^{\infty}_b ({\mathbb R}^d; {\mathbb R}^d)$ and $\tilde f \in C^{\infty}_b ({\mathbb R}^d)$. 
 
To construct the solution we use a probabilistic method (for an alternative vanishing viscosity method see Section 3.2 in \cite{Si}). Let $(X_t^x)$ be the  solution of $dX_t = b(X_t) dt + dL_t$, $X_0 =x \in {\mathbb R}^d$ and consider the associated Markov semigroup $(R_t)$, i.e.,
$R_t l (x)$ $= E [l(X_t^x)]$, $t \ge 0,$ $x \in {\mathbb R}^d$, $l \in UC_b({\mathbb R}^d)$ ($UC_b({\mathbb R}^d) \subset C_b({\mathbb R}^d)$ denotes  the Banach space of all uniformly continuous and bounded functions endowed with the sup-norm). Differentiating with respect to $x$ under the expectation (using the derivative of $X_t^x$ with respect to $x$, cf. \cite{Z}) it is straightforward to prove that $R_t g \in C^1_b({\mathbb R}^d)$, for any $t \ge 0$ and $g \in C^1_b({\mathbb R}^d)$.
For the given $\tilde f\in C^{\infty}_b ({\mathbb R}^d)$ we define 
\begin{equation} \label{ww1}
\tilde w(x) = \tilde w_{\lambda}(x) = \int_0^{\infty} e^{- \lambda t} R_t \tilde f(x)dt, \;\;\; x \in {\mathbb R}^d.
\end{equation}
It is clear that $\tilde w \in C_b({\mathbb R}^d)$. We now show that $\tilde w \in C^1_b({\mathbb R}^d)$ and solves our equation.
To this purpose we first prove that for $t>0$
 \begin{equation}
\label{stimo1}
\sup_{x \in {\mathbb R}^d } |DR_t \tilde f(x)| \le c(\alpha, \beta, \|D b\|_0) \,\,    { (t\wedge1)^{(\beta - 1) / \alpha} } \, \| \tilde f\|_{C^{\beta}_b({\mathbb R}^d)}.
\end{equation} 
Once  this estimate
is proved,  differentiating under the integral sign in \eqref{ww1} we obtain that     $w \in C^1_b({\mathbb R}^d)$ since $\alpha + \beta > 1$. 
Let us fix $t \in (0,1].$
 By Theorem 1.1 in \cite{Z} we know in particular that 
 $$ \| DR_t g\|_0 =
\sup_{x \in {\mathbb R}^d } |DR_t g(x)| \le {c(\alpha) \, e^{\| Db\|_0}}\,\,  {t^{- 1 / \alpha}} \, \| g\|_0,\;\;\; g \in C^1_b({\mathbb R}^d).
$$
 Using the total variation norm   as in  Lemma 7.1.5 of \cite{DZ} we 
deduce  that $R_t l$ is Lipschitz continuous for any $l \in UC_b({\mathbb R}^d)$ and moreover $|R_t l(x) - R_t l(y) | $ $\le {c(\alpha) \, e^{\| Db\|_0}}\,\, $ $ {t^{- 1 / \alpha}} \,|x-y|\, \| l\|_0,$ $x,y \in {\mathbb R}^d.$ 
By Theorem 1.1 in \cite{Z}, for any $g \in C^1_b({\mathbb R}^d)$,   we can  write   the directional derivative of $R_t g$ along $h \in {\mathbb R}^d$ as follows:
\begin{equation}
\label{zhang}
D_h R_t g(x) = E [g(X_t^x)\, J(t,x,h)],\;\; x \in {\mathbb R}^d,
\end{equation} 
where $J(t,x,h)$ is a suitable random  variable such that $(E|J(t,x,h)|^2 )^{1/2}$ $\le c(\alpha) \,$ $ e^{\| Db\|_0}\,\,  {t^{- 1 / \alpha}} |h|$, for any $x \in {\mathbb R}^d$.
Let again $l \in UC_b({\mathbb R}^d)$. Using mollifiers we can consider  an approximating sequence $(g_n) \subset C^{\infty}_b({\mathbb R}^d)$ such that $\| g_n - l\|_0 \to 0$ as $n \to \infty$. Using  \eqref{zhang}  when $g$ is replaced by  $g_n$ and passing to the limit it is not difficult to prove that $R_t  l \in C^1_b({\mathbb R}^d)$ and moreover \eqref{zhang} holds  when $g$ is replaced by $l$ (cf. page 480 in  \cite{PZ}).  

We have found that $R_t : UC_b({\mathbb R}^d) \to  C^1_b({\mathbb R}^d)$  is a linear and bounded operator and 
$ |DR_t l(x)| \le {c(\alpha) \, e^{\| Db\|_0}}\,\,  {t^{- 1 / \alpha}} $ $ \| l\|_0,$ for $x \in {\mathbb R}^d,$ $l \in UC_b({\mathbb R}^d).$ Moreover, 
$R_t : C_b^1({\mathbb R}^d) \to  C^1_b({\mathbb R}^d)$  is linear and bounded and 
$ |DR_t g(x)| \le e^{\| Db\|_{0}} \| Dg \|_0$, for $x \in {\mathbb R}^d,$ $g \in C^1_b({\mathbb R}^d)$.  To prove such estimate we fix $h \in \R^d$ and differentiate $R_t g(x)$ with respect to $x$ along the direction $h$.  One can 
show  that
\begin {equation} \label{z1}
D_{h} E[g(X_t^x)] = E [Dg(X_t^x) \eta_t]
\end{equation}
where  $ \eta_t = D_h X_t^x$ solves $\eta_t = h + \int_0^t Db (X_s^x) \eta_s ds$, $t \ge 0,$ $P$-a.s.. Note that  $|D_h X_t^x| \le |h| e^{\| Db\|_0 t}$ by the Gronwall lemma
 (cf.  page 1211 in \cite{Z}).

By interpolation techniques we know that $\big (UC_b ( {\mathbb R}^d), C_b^1 ({\mathbb R}^d) \big)_{\beta, \, \infty}
$ $ = C_b^{\beta} ({\mathbb R}^d),$  for  $\beta \in (0,1)$
 (cf. \cite[Chapter 1]{L} and the proof of Theorem 3.3 in \cite{Pr10});  it follows that for any $t \in (0,1]$ we have that 
 $R_t : C_b^{\beta}({\mathbb R}^d) \to  C^1_b({\mathbb R}^d)$  is linear and bounded and 
$ |DR_t f(x)| \le {c(\alpha, {\beta}) \, e^{\| Db\|_0}}\,\,  {t^{(\beta- 1) / \alpha}} $ $ \| f\|_{C^{\beta}_b},$ for any $x \in {\mathbb R}^d,$ $f \in C_b^{\beta}({\mathbb R}^d).$ 

We have verified  \eqref{stimo1}  when $t \in (0,1]$. 
If $t>1$ we use a standard argument based on the semigroup property and get,  for any  $x \in {\mathbb R}^d,$
 $|DR_t \tilde f(x)|= |D R_{1} (R_{t-1} \tilde f)(x)|$   
 $\le {c(\alpha) \, e^{\| Db\|_0}} \, \| R_{t-1} \tilde f\|_0 $ $\le {c(\alpha) \, e^{\| Db\|_0}}\,  \|\tilde f\|_0.$
 Thus  \eqref{stimo1} holds and we know that   $\tilde w \in C^1_b({\mathbb R}^d)$. 
To prove that $\tilde w$ is a solution we  first establish the identity
\begin{equation} \label{timed}
\partial_t (R_t \tilde f)(s,x) = R_s ({\cal L}_b \tilde f)(x) = {\cal L}_b (R_s \tilde f)(x), \;\; s \ge 0, \; x \in {\mathbb R}^d.
\end{equation} 
 By using Ito's formula (see  \cite[Section 2.3]{Ku}) and taking the expectation we find
$
E[\tilde f (X_{s+h}^x)] - E[\tilde f (X_{s}^x)] $ $ = \int_s^{s+h} E[ ({\cal L}_b \tilde f)(X^x_r)] dr, 
$ for $ h \in {\mathbb R}$ such that $s +h >0$.
  It follows that, for $x \in {\mathbb R}^d,$
\begin{gather}
\label{diro0}
\partial_t (R_t \tilde f)(s,x) = \lim_{h \to 0}
{h^{-1}}{ \big(R_{s+h} \tilde f (x) - R_{s} \tilde f (x) \big) }  = R_s ({\cal L}_b \tilde f)(x),\; s>0,
\\ \text{and } \;\;
\label{diro}
\lim_{h \to 0^+}
{h^{-1}}{ \big( R_{h} \tilde f (x) -  \tilde f (x) \big) } = {\cal L}_b \tilde f(x).
\end{gather} 
If $s>0$ by \eqref{diro} we get
 $\lim_{h \to 0^+}
\frac{ R_{h} (R_s \tilde f) (x) - R_{s} \tilde f (x)  }{h} = 
{\cal L}_b ( R_s\tilde f)(x)$ when $\tilde f$ in \eqref{diro} is replaced by $R_s \tilde f$. By    the semigroup law,  the last limit and \eqref{diro0} coincide and so  \eqref{timed} holds.
To check that $\tilde w$ verifies  $\lambda \tilde w - {\cal L}_b \tilde w= \tilde f$ we use \eqref{stimo1} and \eqref{timed}. 
 First by the Fubini theorem we have
$$
{\cal L}_b \tilde w (x) = \int_0^{\infty} e^{- \lambda t} {\cal L}_b (R_t \tilde f)(x)dt =  \int_0^{\infty} e^{- \lambda t} R_t ( {\cal L}_b\tilde f)(x)dt.
$$
By \eqref{timed} it follows that, for any $x \in {\mathbb R}^d$,
 $
{\cal L}_b \tilde w (x) = \int_0^{\infty} e^{-\lambda t} \frac{d}{dt} (R_t 
\tilde f(x))dt.
$
Integrating by parts, we get the assertion.

\hh \textit{III step.} Let $\lambda \ge 1$. We  prove the existence of  a $C^{\alpha+ \beta }_b$-solution to  $\lambda w - {\cal L}_b w=f$ on ${\mathbb R}^d$ when $b \in C^{\beta}_b ({\mathbb R}^d; {\mathbb R}^d)$ and $f \in C^{\beta}_b ({\mathbb R}^d)$,  $\alpha + \beta >1$, and show \eqref{sch44}.

Using convolution with mollifiers  and
   possibly  passing to  subsequences (see, for instance, page 431 in \cite{Pr10}) 
   one can consider operators ${\cal L}_{b_n}$ with 
drifts $b_n \in C^{\infty}_b ({\mathbb R}^d; {\mathbb R}^d)$   such that  
   $\| b_n \|_{C^{\beta}_b}$ $\le 
 $ $\| b \|_{C^{\beta}_b}$, $n \ge 1$, and   
 $b_n \to b $  $\;\; \text{in}$ $ \;\; C^{\beta'}(K; {\mathbb R}^d)$
 for any compact set $K \subset {\mathbb R}^d$ and $\beta' \in (0,  \beta)$. Similarly one can construct $(f_n) \subset C_b^{\infty}({\mathbb R}^d)$ such that  $\| f_n \|_{C^{\beta}_b}$ $\le 
 $ $\| f \|_{C^{\beta}_b}$, $n \ge 1$, and  
  $f_n \to f $  {in} $C^{\beta'}(K)$ for any compact set $K \subset {\mathbb R}^d$ and $\beta' \in (0,  \beta)$.
 By II step  there exist $C^1_b$-solutions $w_n $ to  ${\cal L}_{b_n} w_n = \lambda w_n - f_n $, $n \ge 1$.
By step I we know that 
 $w_n \in  C^{\alpha + \beta}_b({\mathbb R}^d)$, $n \ge 1$, with  the  estimate  
 \begin{equation} \label{silvv2}
 \|w_n \|_{C^{\alpha + \beta}_b({\mathbb R}^d)} \le 2  C_3 
  \|f \|_{C^{\beta}_b({\mathbb R}^d)},
\end{equation}
 ($C_3 = C_3 (\| b \|_{C^{\beta}_b}, \alpha, \beta, d)$ is independent of $\lambda$ and $n$).
 Possibly
passing to \ {a subsequence}  still denoted with $(w_n)$,
   we have that
  $
 w_n \to w $  in $  C^{\alpha + \beta'}(K),$
 for any compact set $K \subset {\mathbb R}^d$ with ${\beta}' >0$
  such that $1 < \alpha +
 {\beta}' < \alpha +\, \beta.$ Moreover, \eqref{silvv2} holds with $w_n$ replaced by $w.$ 
   We can easily
 pass to the limit in each term of 
  $ \lambda w_n(x)- {\cal L} w_n(x) - b_n(x)
   \cdot Dw_n(x)$ $ = f_n(x)$ as $n \to \infty$
 and obtain that $w$ solves our equation.

\hh  \textit{IV step.} We prove the final assertion.

We already know that there exists a unique solution $w_{\lambda} 
 \in C^{\alpha + \beta}_b({\mathbb R}^d)$ and that \eqref{sch44} holds.
To complete the proof we argue as in the final part of the proof of Theorem \ref{reg}. By 
    the interpolatory estimate   
$ \| Dw_{\lambda}\|_{0} $ $\le N(\alpha, \beta, d)
[Dw_{\lambda}]_{C^{\alpha + \beta - 1}_b}^{\frac{1}{\alpha + \beta}}
\, \| w_{\lambda} \|_{0}^{ \frac{\alpha + \beta -1}{\alpha + \beta}} $, 
we obtain easily that 
  $ \| Dw_{\lambda}\|_{0} 
   < 1/3$ for    $\lambda \ge \lambda_0(d, 
\| b\|_{C^{\beta}_b}, \alpha , \beta) $.
  \end{proof}

\begin{proof}[Proof of Theorem \ref{hp}] As in  the proof of
Theorem \ref{main} we  verify  the assumptions of  Theorem \ref{d32}. Note that Hypothesis \ref{zero} holds since 
$ \int_{\{ |x| > 1  \, \}} \frac{|y|^{\theta}} 
{|y|^{d+ \alpha}} dy < \infty,
 $ for any $\theta \in (0, \alpha).$ 
 In order to check Hypothesis \ref{primo} we argue as in the proof of Theorem \ref{main} (using  Theorems \ref{sil} and
\ref{uno11}; recall that 
 $\alpha = \alpha_0$).  The proof is complete.
\end{proof}

\vskip 8mm

\noindent \textbf{Acknowledgement.}
The author would like to thank the anonymous  referees
 for their
 useful comments and suggestions.

\end{document}